\documentclass[aos,preprint]{imsart}

\RequirePackage{amsthm,amsmath,amsfonts,amssymb}
\RequirePackage[authoryear]{natbib} 
\RequirePackage[colorlinks,citecolor=blue,urlcolor=blue]{hyperref}
\RequirePackage[dvipdfmx]{graphicx}

\startlocaldefs
\theoremstyle{plain}
\newtheorem{thm}{Theorem}
\newtheorem{cor}[thm]{Corollary}
\newtheorem{lem}[thm]{Lemma}
\newtheorem{pro}[thm]{Proposition}
\theoremstyle{remark}
\newtheorem{rem}{Remark}
\newtheorem{defi}[rem]{Definition}
\newtheorem{exa}[rem]{Example}
\newtheorem{alg}[rem]{Algorithm}
\newcommand{\lbar}[1]{ \overline{#1} }

\newcommand{\Tr}[1]{ \mathrm{Tr}\left[#1\right]}
\newcommand{\Det}[1]{ \mathrm{Det}\left(#1\right)}
\newcommand{\set}[2]{ \left\{\,#1\,;\,#2\,\right\} }
\newcommand{\cpling}[2]{ \langle #1,#2 \rangle}
\newcommand{\transp}[1]{ #1^{\top}}
\newcommand{\dd}{\,\mathrm{d}}
\newcommand{\E}{\mathbb{E}}
\newcommand{\A}{\mathcal{A}}
\newcommand{\nc}{\newcommand}
\nc{\Pcone}{\mathcal{P}}
\nc{\Graph}{\mathcal{G}}
\nc{\Zsp}{\mathcal{Z}}
\nc{\iu}{i}
\nc{\integer}{\mathbb{Z}}
\nc{\real}{\mathbb{R}}
\nc{\complex}{\mathbb{C}}
\nc{\quat}{\mathbb{H}}
\nc{\K}{\mathbb{K}}
\endlocaldefs

\begin{document}

\begin{frontmatter}
\title{Model selection in the space of Gaussian models invariant by symmetry}
\runtitle{Models invariant by symmetry}

\begin{aug}
\author[A]{\fnms{Piotr} \snm{Graczyk}\ead[label=e1]{piotr.graczyk@univ-angers.fr}},
\author[B]{\fnms{Hideyuki} \snm{Ishi}\ead[label=e2]{hideyuki@sci.osaka-cu.ac.jp}}\thanksref{t2},
\author[C]{\fnms{Bartosz} \snm{Ko{\l}odziejek}\ead[label=e3]{b.kolodziejek@mini.pw.edu.pl}}\thanksref{t3}
\\\and
\author[D]{\fnms{H\'{e}l\`{e}ne} \snm{Massam}\ead[label=e4]{massamh@mathstat.yorku.ca}}\thanksref{t4}
\address[A]{Universit\'e d'Angers, France, \printead{e1}}
\address[B]{Osaka City University, Japan, \printead{e2}}
\address[C]{Warsaw University of Technology, Poland, \printead{e3}}
\address[D]{York University, Canada, \printead{e4}}
\thankstext{t2}{Supported by JSPS KAKENHI Grant Number 16K05174 and JST PRESTO.}
\thankstext{t3}{Supported by grant 2016/21/B/ST1/00005 of the National Science Center, Poland. }
\thankstext{t4}{Supported by an NSERC Discovery Grant.}
\end{aug}

\runauthor{Graczyk, Ishi, Ko{\l}odziejek  and Massam}

\begin{abstract}
	We consider multivariate centered
Gaussian models for the random variable $Z=(Z_1,\ldots, Z_p)$, invariant under  the action of a subgroup of the group of permutations on $\{1,\ldots, p\}$. Using the representation theory of the symmetric group on the field of reals, we derive the 
distribution of the maximum likelihood estimate of the covariance parameter $\Sigma$ and also the analytic expression of the normalizing constant of the Diaconis-Ylvisaker conjugate prior for the precision parameter $K=\Sigma^{-1}$. We can thus perform Bayesian model selection in the class of complete Gaussian models invariant by the action of a subgroup of the symmetric group, which we could also call complete RCOP models. We illustrate our results with a toy example of dimension $4$ and several examples for selection within cyclic groups, including a high dimensional example with $p=100$.
\end{abstract}



\end{frontmatter}


\section{Introduction}
\subsection{Motivations and applications}

Let $V=\{1,\ldots,p\}$ be a finite index set and let $Z=(Z_1,\ldots, Z_p)$ be a multivariate random variable following a centered Gaussian model $\mathrm{N}_p(0,\Sigma)$.
Let $\mathfrak{S}_p$ denote the symmetric group
on $V$, that is, the group of all permutations on $\{1,\ldots,p\}$ and let $\Gamma$ be a subgroup of $\mathfrak{S}_p$. A centered Gaussian model is said to be invariant under the action of $\Gamma$ if for all $g\in \Gamma$, $g\cdot\Sigma\cdot\transp{g}=\Sigma$ 
(here we identify a permutation $g$ with its permutation matrix).

Given $n$ data points $Z^{(1)},\ldots,Z^{(n)}$ from a Gaussian distribution, our aim  in this paper is to do Bayesian model selection within the class of models invariant by symmetry, that is, invariant under the action of some subgroup $\Gamma$ of $\mathfrak{S}_p$ on $V$. Given the data, our aim is therefore to identify the subgroup $\Gamma\subset \mathfrak{S}_p$ such that the model invariant under $\Gamma$ has the highest posterior probability. We achieve this goal by constructing a Markov chain on the space of models and using the Metropolis-Hastings algorithm.

There are many alternative ways of doing model search in modern statistics on big data sets, both frequentist and Bayesian. Bayesian model selection methods (cf. \cite[Ch.10]{Ghosal}) are widely used in practice, thanks to the possibility of using a prior knowledge on the model and to their  rigorous  mathematical bases.
Moreover, in Bayesian approach the Metropolis-Hastings algorithm is naturally applicable and generally
accepted. 

Our work can be viewed as a special case of  colored graphical Gaussian models (the underlying graph is complete so we do not impose conditional independence structure), that is, statistical graphical models with additional symmetries (equality constraints on the precision or correlation matrix). Such models were
introduced into modern exploratory analysis of data in the seminal paper
\cite{HL08},
as a powerful  tool of dimension reduction in unsupervised learning, cf. \cite[Chapter 9.8]{NewBook}.
A preponderant role is given in  \cite{HL08} to the RCOP models studied in our paper,
thanks to their most tractable structure and interpretability through symmetries among the variables. One of motivations of this paper is to address the task stated in \cite[p. 1025]{HL08}:
\emph{For the models to become widely applicable, it is mandatory to develop algorithms for model identification which are robust, reliable and transparent}.

For high-dimensional data, Gaussian models which
have symmetries and are graphical allow statisticians to reduce the dimension of a model. In genetics, such models can be used to identify genes having the same function or groups of genes
having similar interactions. Below we mention some of the studies in which our model could find potential application.

In \cite{HL08}, gene expression signatures for p53 mutation status in $58$ breast cancer samples consisting of $150$ genes were investigated and interpreted. We apply our algorithm to this data in the Supplementary material, see Section 6 in \cite{GIKM_SM}. We claim that our model
selection procedure can be used as an exploratory tool. Assuming that the
variables are all on some common scale, the proposed algorithm can be run
to look for potential hidden symmetries between the variables.

It is worth to underline that one of the characteristics of our model is the lack of scale invariance. We point out below that there are many examples where our model can still be applied.
E.g. the data from gene  expression  are on the  same scale in the sense that they are results of experiments of the same type, measured in the same gauges. Similar situation appears  generally for omic data sets in proteomics and metabolomics. For more details see e.g. the monograph  \cite{med4}.
In \cite[Section 6.2 TCGA Breast Cancer Data]{med5},
genetic information in tumoral tissues DNA that are involved in gene expression are measured from messenger sequencing 	by the RNASeq method and they are all on the same scale, as they are the numbers of transcripts in a sample.
In clinical epidemiology and medicine, one often uses scales  combined into scores to classify outcome, see e.g. \cite{med1,med2}. 
Range of values of such scores are often similar, even though not formally tested
statistically to be so. 
In the paper \cite{med3} the normalization or non-normalization of data did not influence	their statistical interpretation.

Moreover, we argue that it is natural to expect certain symmetries in the data from gene expression. Namely, expression of a given gene is triggered by binding the transcription factors to the gene transcription factor binding sites. The transcription factors are the proteins produced by other genes, say regulatory genes. In the gene network there are often many genes triggered by the same regulatory genes and it makes sense to assume that their  relative expressions depend on the abundance of proteins of the regulatory genes (i.e. gene expressions) in a similar way. 

In \cite{GaoM15},  $12\,625$ neutrophil gene expressions were monitored 
with imposed symmetry constraints to  the  graphical  modeling. The paper \cite{Hel20} contains a study
of  the structure of colored graphs applied to a flow cytometry data set on signaling networks of human immune system cells, which consists of $7\,466$ measurements on $11$ phosphorylated proteins.

A very recent application  of graphical  models  with  symmetries to fMRI real data on brain networks is proposed
in \cite{Ro21}.
An impressive number of  recent applications of 
graphical models to real data analysis is listed in the recent monograph
\cite[Chapters 19,20,21] {NewBook} and includes
genetics, genomics,  molecular systems biology and
forensic analysis, cf. also the books \cite{Ro2017}
for medical and  \cite{Li09} for image data applications.

Finally, let us mention that colored graphical models provide interesting examples of
exponential algebraic varieties and
algebraic exponential families, e.g. Toeplitz matrices
\cite{Mi16};  see also \cite{DaMa21}. The recent algebro-geometric approach to graphical models and Gaussian Bayesian networks is being developed intensely \cite[Chapter 3]{NewBook}.

\subsection{Contribution of the paper and relations to previous work}
In this subsection we carefully describe and position 
this paper  in the context of previous research.

Theory of invariant normal models (with the so-called lattice conditional independencies \cite{AM98,Ma00}, which are not considered in the present paper) was developed by the Danish school. History regarding this subject is nicely presented in \cite{AM98}, where the reader can also find references to earlier works dealing with particular symmetry models such as, for example, the circular symmetry model of \cite{OP69} that we will consider further (Section \ref{practical}). 
Among others, Andersson, Br\o ns, Jensen, Madsen and Perlman developed a fairly complete theory of MLE $\hat{\Sigma}$ of the covariance matrix in invariant normal models, however, the problems considered in our paper are very different.

These works were concentrating on the 
derivation of statistical properties of the maximum likelihood estimate of $\Sigma$ and on testing the hypothesis that models were of a particular type.
 In particular, to the best of our knowledge the Danish school never considered 
any model search in the context of invariant normal models. When the model space is very big (and this is the usual case of our framework), then it is impossible to perform simultaneous tests for all possible models. {Despite the computation problems, there is also even bigger issue due to multiple comparisons problem.}

Just like the classical papers mentioned above,  the fundamental algebraic tool we use in this work is the irreducible decomposition theorem for the matrix representation of 
the group $\Gamma$,
which in turn means that, through an adequate change of basis, any matrix $X$ in $\Zsp_{\Gamma}$, the space of symmetric matrices invariant under the subgroup $\Gamma$ of $\mathfrak{S}_p$, can be written in a block diagonal form. The following result is a reformulation of an observation made in \cite[4.6-4.8]{A75}. 
\begin{thm}\label{thm:intro}
	Fix a permutation subgroup $\Gamma\subset\mathfrak{S}_p$. Then, there exist constants $L\in\mathbb{N}$, $(k_i, d_i, r_i)_{i=1}^L$ and orthogonal matrix $U_\Gamma$ such that if $X\in\mathcal{Z}_\Gamma$, i.e. $X\in\mathrm{Sym}(p;\real)$ and
	\[
	X_{ij}=X_{\sigma(i)\sigma(j)}\qquad(\sigma\in\Gamma,\, i,j\in\{1,\ldots,p\}),
	\]
	then
	\begin{equation}
		\label{genform}
		X
		= U_\Gamma\cdot
		\begin{pmatrix} 
			M_{\mathbb{K}_1}(x_1) \otimes I_{k_1/d_1} & & & \\
			& M_{\mathbb{K}_2}(x_2) \otimes I_{k_2/d_2} & & \\
			& & \ddots & \\
			& & & M_{\mathbb{K}_L}(x_L) \otimes I_{k_L/d_L} 
		\end{pmatrix}\cdot \transp{U_\Gamma},
	\end{equation}
	where $M_{\mathbb{K}_i}(x_i)$ is a real matrix representation of 
	an $r_i \times r_i$ Hermitian matrix $x_i$ with entries in  $\mathbb{K}_i=\real, \complex$ or $\quat$, $i=1,\ldots, L$,
	and $A \otimes B$ denotes the Kronecker product of matrices $A$ and $B$. 
\end{thm}
Elements of $(k_i, d_i, r_i)_{i=1}^L$  are integer constants called structure constants that we will define later. At this point we note that $k_i/d_i$ are also integers and $d_i = \dim_{\real} \mathbb{K}_i \in\{1,2,4\}$. The mappings $M_{\mathbb{K}_i}\colon \mathrm{Herm}(r_i;\mathbb{K}_i)\to \mathrm{Sym}(d_i r_i;\real)$ are defined in Section \ref{sec:Jordan}.
As was already observed in \cite{Je88}, the space $\Zsp_\Gamma$ equipped with a Jordan product and trace inner product forms a Euclidean Jordan algebra. Thus, \eqref{genform} can be understood as a decomposition of $\mathcal{Z}_\Gamma$ into Euclidean simple Jordan algebras. Theorem \ref{thm:intro} is the existence result and actual computation of structure constants and the orthogonal matrix $U_\Gamma$ is in general a hard technical task. A complete proof of Theorem \ref{thm:intro} can be found in the Supplementary material \cite{GIKM_SM}. We tried to ensure that our arguments are concrete and should be easier to understand for the reader who is not familiar with representation theory. 

The main novel results of the paper are

\indent  \textbf{(a)} new Bayesian model selection procedure within Gaussian models invariant by a permutation subgroup, Section \ref{selection:cyclic},\\
\indent \textbf{(b)} explicit formulas for Gamma integrals, normalizing constants of densities of Diaconis-Ylvisaker conjugate prior for $K$ and of the MLE of $\Sigma$ on $\Pcone_{\Gamma}$, Bayes factors, which are necessary for performing a), Theorems \ref{thm:IntegralPhi} and \ref{thm:IntegralNoPhi} in Section \ref{gammaintegral},\\
\indent \textbf{(c)} efficient algorithm for finding a decomposition \eqref{genform} when the subgroup $\Gamma$ is cyclic, Theorems \ref{thm:cyclic1} and \ref{thm:cyclic2} in Section \ref{cyclic:basic},\\
\indent \textbf{(d)} simulations that visualize the performance of the method in low and high dimensional examples,  Section \ref{selection:general}, Section \ref{practical} and Section 4 of the Supplement \cite{GIKM_SM}.

\textbf{Ad (a). } 
We are aware of three papers which concern model selection in the space of colored graphical model, namely \cite{Ge11,Hel18b,Hel20}.

In \cite{Ge11} the author used the lattice structure of the colored graphical model classes and applied Edwards–Havr\'anek model selection procedure to $p=4$ and $p=5$ examples, admitted that applying this method to high-dimensional problems requires additional work.

Both papers \cite{Hel18b} and \cite{Hel20} used Bayesian methods and allow for model selection in the space of RCON models (which is a superclass of RCOP models introduced in \cite{HL08}) and for arbitrary graphs describing conditional independencies in a vector. Such generality comes at a certain cost: as the authors were not able to compute normalizing constants for such general models, they had to approximate these constants or bypass the problem (which comes with a significant increase in computational complexity):  we quote a few lines from these articles that describe the situation well. 
\begin{itemize}
	\item \cite{Hel18b}: \textit{However, just as sampling schemes for the $G$-Wishart distribution are not recommended for computation of (normalizing constant) $I_G(\delta,D)$  and model selection in higher dimensions, our sampling scheme is not recommended for computing (normalizing constant) $I_G(\delta,D)$ in high dimensions.}
	\item \cite{Hel20}: \textit{The model $G^\ast$ with an additional edge is then compared to the current model $G$ using the Bayes factor (...) which itself is computed with the help of the double
		reversible jump MCMC algorithm.  (...) We thus avoid computing these quantities which are the usual computational stumbling blocks in graphical Gaussian model selection.}
\end{itemize}
Our approach to the Bayesian model selection is much simpler as we were able to compute normalizing constants of Diaconis-Ylvisaker conjugate priors for $K$. 

\textbf{Ad (b). }
We note that a general form of a density of the MLE under our assumptions was already written in \cite{A75} and in more explicit form in \cite{AM98}.
However, an explicit expression for the normalizing constant of density of $\hat{\Sigma}$ or Diaconis-Ylvisaker conjugate prior  was not the object of interest of the Danish school
and it is crucial for the Bayes paradigm and the Bayesian model selection.

Still, there are certain results in their numerous works that can be compared with our formulas. In particular, \cite[Eq. (A.4)]{AM98} gives a formula for $\mathbb{E}[\mathrm{Det}(\hat{\Sigma})^\alpha]$ which is consistent with our results. 
Indeed, after substitution of $(d_i, n\,k_i/d_i, r_i)_i$ for $(d_\mu,n_\mu,p_\mu)_\mu$, the right hand side of their formula coincides with $2^{\alpha p}n^{-\alpha p} \Gamma_{\Pcone_\Gamma}(\alpha+n/2) /\Gamma_{\Pcone_\Gamma}(n/2)$ in our notation (see Theorem \ref{thm:IntegralPhi}). 
Further, in \cite[Section 8]{ABJ83}  explicit formula for normalizing constants of the density of eigenvalues of $\hat{\Sigma}$ is given. However, as distribution of eigenvalues of a random matrix does not determine the distribution of this matrix, our formulas do not follow from these results. 

In some very special cases, normalizing constants for Diaconis-Ylvisaker conjugate prior are given in \cite{Hel18b}.

\textbf{Ad (c).} In order to compute normalizing constants in our model, one needs to know explicit decomposition \eqref{genform}, that is, the structure constants and the orthogonal matrix $U_\Gamma$. The same issue can be seen in \cite[Theorem 1]{Je88}, which is the existence result (like our Theorem \ref{thm:intro}) and does not give the answer how should one proceed to find such decomposition.  In order for this theory to be applied, we proved that when $\Gamma$ is a cyclic subgroup, then we can efficiently find explicit decomposition  \eqref{genform} for arbitrary $p$. 
This practical aspect of our work has not been addressed before.  
To our knowledge, our paper is the first one to identify a non-trivial class of subgroups for which all objects can be calculated explicitly. 

For a moment, let us consider the more general situation of Gaussian graphical models with conditional independence structure encoded by a non complete graph $G$. Then one can introduce symmetry restrictions (RCOP) by requiring that the precision matrix $K$ is invariant under some subgroup $\Gamma$ of $\mathfrak{S}_p$. However, when $G$ is not complete, not all subgroups are suited to the problem. In such cases, one has to require that $\Gamma$ belongs to the automorphism group $\mathrm{Aut}(G)$ of $G$. If a graph $G$ is sparse, then $\mathrm{Aut}(G)$ may be very small and it is natural to expect that the vast majority of subgroups of $\mathrm{Aut}(G)$ are actually cyclic. Moreover, finding the structure constants for a general group is much more expensive and in some situations it may not be worth to consider the problem in its full generality.
We consider our work as a first step towards the rigorous analytical treatment of Bayesian model selection in the space of graphical Gaussian models invariant under the action of $\Gamma\subset \mathfrak{S}_p$ when conditional independencies are allowed.

Moreover, we offer here a new heuristic approach to colored graphical models using our ``full graph'' approach. It was already observed in \cite{HL08} that the color pattern of the covariance matrix and the precision matrix are the same (i.e. they belong to the space $\mathcal{Z}_\Gamma$). The same applies to the off-diagonal elements of the partial correlation matrix. Our procedure allows one to find the color pattern of the covariance matrix. Since our model does not suppose any preliminary conditional independence structure, the corresponding graph is complete and there are no zeros in the partial correlation matrix. However, if the true graph is not complete, it is natural to expect from the model that similar entries of the partial correlation matrix (in particular those which are close to $0$) are colored in the same way. Thus, to recover the true graph we may threshold the values of the partial correlation matrix. More precisely, we choose a threshold $\alpha>0$ and we construct a colored graph $G$ by maintaining the color pattern previously found and requiring that for $i\neq j$,
\[
i\sim j \qquad	\mbox{if and only if }\qquad\frac{|k_{ij}|}{\sqrt{k_{ii}k_{jj}}}>\alpha,
\] 
where $K=(k_{ij})_{i,j}$ is the precision matrix.
Resulting graph $G$ is in general not complete and the corresponding space of admissible covariance matrices is still invariant under the action of the subgroup found by our procedure; thus we obtain a RCOP model. 
We applied this approach to a real data example in Section 4 of the Supplementary material. 


There are also several recent papers which use a version of Theorem \ref{thm:intro}. The subject of \cite{STW16} is estimation of complex covariance matrices in complex random vectors in non-Gaussian models invariant under the action of a fixed permutation subgroup, see also \cite{IEE2}. We remark that the argument of \cite{STW16} is based on representation theory over complex number fields, and as was noticed by them, the fundamental structure theorem is much simpler than Theorem \ref{thm:intro}  because of the difference between the representation theory over $\complex$ and $\real$.  In \cite{Sh13} the authors consider the real case and sub-Gaussian model for which they establish rates of convergence of an estimator of $\Sigma$, empirical covariance matrix regularized by the action of a known permutation subgroup.

\subsection{Outline of the paper}

Let us consider the following example, which shows how Theorem \ref{thm:intro} works.
\begin{exa}
	\label{ex1}
	For $p=3$ and $\Gamma=\mathfrak{S}_3$, the space of symmetric matrices $X$ invariant under $\Gamma$, that is, such that $X_{ij}=X_{\sigma(i)\sigma(j)}$ for all $\sigma\in \Gamma$, is
	\[\Zsp_{\Gamma} = \set{ \begin{pmatrix}
			a & b & b \\ b & a & b \\ b & b & a
	\end{pmatrix}}{a,b\in\real}.\] 
	
	The decomposition \eqref{genform} yields
	$U_\Gamma := \begin{pmatrix} v_1, & v_2, & v_3 \end{pmatrix} \in \mathrm{O}(3)$ with
	$$
	v_1 := \begin{pmatrix} 1/\sqrt{3} \\ 1/\sqrt{3} \\ 1/ \sqrt{3} \end{pmatrix}, \quad
	v_2 := \begin{pmatrix} \sqrt{2/3} \\ - 1/\sqrt{6} \\ -1/\sqrt{6} \end{pmatrix}, \quad
	v_3 := \begin{pmatrix} 0 \\ 1/\sqrt{2} \\ - 1 \sqrt{2} \end{pmatrix},
	$$
	and
	$$
	\begin{pmatrix} a & b & b \\ b & a & b \\ a & a & b \end{pmatrix} 
	=U_\Gamma \cdot\begin{pmatrix} a + 2b & & \\ & a-b & \\ & & a-b \end{pmatrix}\cdot \transp{U_\Gamma},
	$$
	Here $L=2$, $k_1/d_1=1$, $k_2/d_2=2$, $\K_1=\K_2=\real$, $d_1=d_2=1$. 
\end{exa}

We see immediately in the example above that, following the decomposition \eqref{genform}, the trace $\Tr{X}=a+2b+2(a-b)$ and  the determinant $\Det{X}=(a+2b)(a-b)^2$ can be readily obtained. Similarly, using \eqref{genform} allows us to easily obtain  $\Det{X}$ and $\Tr{X}$  in general. 

In Section \ref{gammaintegral}, we will see that having the explicit formulas for
$\Det{X}$ and $\Tr{X}$, in turn, allows us to derive the
analytic expression of
the Gamma function on $\Pcone_{\Gamma}=\mathcal{Z}_\Gamma\cap \mathrm{Sym}^+(p;\real)$,  defined as
\begin{align*}
\Gamma_{\Pcone_\Gamma}(\lambda):=\int_{\Pcone_\Gamma} \Det{X}^\lambda e^{-\Tr{X}} \varphi_\Gamma(X)\dd X,
\end{align*}
where $\varphi_\Gamma(X)\dd X$ is the invariant measure on $\Pcone_\Gamma$
(see Definition \ref{def:phi} and Proposition \ref{pro:phi}) 
and $\dd X$ denotes the Euclidean measure on 
the space $\Zsp_\Gamma$ with the trace inner product. 

With our results,  we can derive  the analytic expression of the normalizing constant $I_{\Gamma}(\delta,D)$ of the Diaconis-Ylvisaker conjugate prior on $K=\Sigma^{-1}$ with density, with respect to the Euclidean measure on $\Zsp_\Gamma$, equal to
\[
f(K;\delta,D)= \frac1{I_{\Gamma}(\delta,D)} {\Det{K}^{(\delta-2)/2} e^{- \tfrac12 \Tr{K \cdot D}} }
{\bf 1}_{\Pcone_{\Gamma}}(K)
\]
for appropriate values of the scalar hyper-parameter $\delta$ and the matrix hyper-parameter $D\in \Pcone_{\Gamma}$. By analogy with the $G$-Wishart distribution, defined in the context of the graphical Gaussian 
models, 
Markov with respect to an undirected graph $G$ on the cone $P_G$ of positive definite matrices with zero entry $(i,j)$ whenever there is no edge between the vertices $i$ and $j$ in $G$,
(see \cite{NewBook}), 
we can call the distribution
with density $f(K;\delta,D)$,
the RCOP-Wishart (RCOP is the name coined in \cite{HL08} for graphical
Gaussian models with restrictions generated
by permutation symmetry).
It is important to note here that if $\Sigma$ is in $\Pcone_{\Gamma}$, so is $K=\Sigma^{-1}$ so that $K$ can also be decomposed according to \eqref{genform}. Equipped with all these results, we  compute the Bayes factors comparing models pairwise and perform model selection. We will indicate in Section \ref{modelselection} how to travel through the space of subgroups of the symmetric group. 

In Section \ref{gammaintegral}, we also derive the distribution of the maximum likelihood estimate (henceforth abbreviated MLE) of $\Sigma$  and show that for $n\geq \max_{i=1,\ldots,L}\left\{r_i d_i/k_i\right\}$ it has a density equal to
\begin{align*}
\frac{\Det{X}^{n/2} e^{- \tfrac12\Tr{X\cdot \Sigma^{-1}}} }
{\Det{2\Sigma}^{n/2} \Gamma_{\Pcone_\Gamma}(\frac{n}{2})} \varphi_\Gamma(X) {\bf 1}_{\Pcone_{\Gamma}}(X).
\end{align*}

Clearly, the key to computing the Gamma integral on $\Pcone_{\Gamma}$, the normalizing constant  $I_\Gamma(\delta,D)$ or the density of the MLE of $\Sigma$  is, for each $\Gamma\subset \mathfrak{S}_p$, to obtain the block diagonal matrix with diagonal block entries $M_{\mathbb{K}_i}(x_i) \otimes I_{k_i/d_i}$, $i=1,\ldots,L$, in the decomposition \eqref{genform}. In principle,
we have to derive the invariant measure $\varphi_\Gamma$  and find the structure constants $(k_i, d_i, r_i)_{i=1}^L$.  This goal can be achieved by constructing an orthogonal matrix $U_{\Gamma}$ and using \eqref{genform}. However, doing so for every $\Gamma$ visited during the model selection process is computationally heavy. 

We will show that for small to moderate dimensions, we can obtain the structure constants  as well as the expression of $\Det{X}$ and $\varphi_{\Gamma}(X)$ without having to compute $U_{\Gamma}$. Indeed, as indicated in Lemma \ref{rem:Useful}, for any $X\in \Pcone_{\Gamma}$, $\Det{X}$ admits a unique irreducible factorization of the form 
\begin{equation} \label{factorization}
\Det{X}  = \prod_{i=1}^L \Det{M_{\mathbb{K}_i}(x_i)}^{k_i/d_i} 
= \prod_{j=1}^{L} f_j(X)^{a_j} \quad (X \in \Zsp_\Gamma),
\end{equation}
where each $a_j$ is a positive integer, 
each $f_j(X)$ is an irreducible polynomial of $X \in \Zsp_\Gamma$, 
and $f_i \ne f_j$ if $i \ne j$.  
The constants $k_i$, $d_i$, $r_i$ are obtained by identification of the two expressions of $\Det{X}$ in \eqref{factorization}.
Factorization of a homogeneous polynomial $\Det{X}$ can be performed using standard software such as either \textsc{Mathematica} or \textsc{Python}. 

Due to computational complexity, for bigger dimensions, it is difficult to obtain the irreducible factorization of $\Det{X}$. For special cases such as the case where the subgroup $\Gamma$ is a cyclic group, we give (Section \ref{cyclic:basic}) a simple construction of the matrix $U_{\Gamma}$
and thus, for any dimension $p$, we can do model selection in the space of models invariant under the action of a cyclic group. We argue that restriction to cyclic groups is not as limiting as it may look. The formula for the number of different colorings $c_p = \#\set{\Zsp_\Gamma}{\Gamma\subset\mathfrak{S}_p}$ for given $p$ is unknown. Obviously, it is bounded from above by the number of all subgroups of $\mathfrak{S}_p$, because different subgroups may produce the same coloring (e.g. in Example \ref{ex1} we have 
$\Zsp_{\mathfrak{S}_3} = \Zsp_{\langle (1,2,3) \rangle}$). On the other hand, it is known (see Lemma \ref{lem:cyc}) that $c_p$ is bounded from below by the number of distinct cyclic subgroups, which grows rapidly with $p$ (see OEIS\footnote{The On-Line Encyclopedia of Integer Sequences, https://oeis.org/.} sequence A051625).
In particular, for $p=18$\footnote{The number of subgroups of $\mathfrak{S}_p$ is unknown for $p>18$, see \cite{Holt} and OEIS sequence A005432.}, we have $c_p \in (7.1\cdot 10^{14},7.6\cdot 10^{18})$, see also Table \ref{tab:tab2}. The lower bound for $c_p$ indicates that the colorings obtained from cyclic subgroups form a rich subfamily of all possible colorings.

The procedure to do model selection will be described in Section \ref{modelselection} and we will illustrate this procedure with Frets' data (see \cite{Frets}) and several examples for selection within cyclic groups, including a high dimensional example with $p=100$ (Section \ref{practical}) and a real data example \cite{Miller} with $p=150$ in Section 4 of the Supplement.



\section{Preliminaries and structure constants}\label{theory}

In this section we present methods to calculate the structure constants of a decomposition given in Theorem \ref{thm:intro}. Additions to this section can be found in Section 3 of the Supplementary material \cite{GIKM_SM}.

\subsection{Notation}
Let $\mathrm{Mat}(n, m; \real)$, $\mathrm{Sym}(n; \real)$ denote the linear spaces of real $n \times m$ matrices and symmetric real
$n \times n$ matrices, respectively. Let $\mathrm{Sym}^+(n; \real)$ be the cone of symmetric positive definite real $n \times n$ matrices. $\transp{A}$ denotes the transpose of a matrix $A$. $\mathrm{Det}$ and $\mathrm{Tr}$ denote the usual determinant and trace in $\mathrm{Mat}(n,n;\real)$.

For $A \in \mathrm{Mat}(m,n;\real)$ and $B \in \mathrm{Mat}(m', n';\real)$,
we denote by $A \oplus B$ the matrix
$\begin{pmatrix} A & O \\ O & B \end{pmatrix} \in \mathrm{Mat}(m+m', n+n';\real)$,
and by $A \otimes B$  the Kronecker product of $A$ and $B$. 
For a positive integer $r$, we write $B^{\oplus r}$ for 
$
I_r \otimes B \in \mathrm{Mat}(rm', rn'; \real)
$

Let $p$ denote the fixed number of vertices of a graph and let $\mathfrak{S}_p$ denote the symmetric group. 
We write permutations in cycle notation, meaning that $(i_1, i_2, \ldots, i_n)$ maps $i_j$ to $i_{j+1}$ for $j=1,\ldots,r-1$ and $i_n$ to $i_1$. By $\langle \sigma_1,\ldots,\sigma_k\rangle$ we denote the group generated by permutations $\sigma_1,\ldots,\sigma_k$. The composition (product) of permutations $\sigma, \sigma'\in\mathfrak{S}_p$ will be denoted by $\sigma\circ\sigma'$.
\begin{defi}
	For a subgroup $\Gamma \subset  \mathfrak{S}_p$,
	we define the space of symmetric matrices invariant under $\Gamma$, or the vector space of colored matrices,
	\[
	\Zsp_{\Gamma} 
	:= \set{x \in \mathrm{Sym}(p;\real)}{x_{ij} = x_{\sigma(i)\sigma(j)} \mbox{ for all }\sigma \in \Gamma},
	\]
	and the cone of positive definite matrices valued in $\Zsp_{\Gamma}$,
	\[
	\label{def:Pgamma}
	\Pcone_{\Gamma} := \Zsp_{\Gamma} \cap \mathrm{Sym}^+(p;\real).
	\]
\end{defi}
We note that the same colored space and cone can be generated by two different subgroups: in Example \ref{ex1}, the subgroup $\Gamma'=\langle (1,2,3)  \rangle$  
generated by the permutation $\sigma= (1,2,3)$
is such that $\Gamma'\neq\Gamma$ but $\Zsp_{\Gamma'}=\Zsp_\Gamma$. Let us define
\[
\Gamma^\ast = \set{ \sigma^\ast\in\mathfrak{S}_p}{x_{ij}=x_{\sigma^\ast(i)\sigma^\ast(j)}\mbox{ for all }x\in\mathcal{Z}_\Gamma}.
\]
Clearly, $\Gamma$ is a subgroup of $\Gamma^\ast$ and $\Gamma^\ast$ is the unique largest subgroup of $\mathfrak{S}_p$ such that $\Zsp_{\Gamma^\ast}=\Zsp_{\Gamma}$ or, equivalently, such that the $\Gamma^\ast$- and $\Gamma$- orbits in $\set{\{v_1,v_2\}}{v_i\in V,\, i=1,2}$ are the same. The group $\Gamma^\ast$ is called  the $2^\ast$-closure of $\Gamma$. The group $\Gamma$ is said to be $2^\ast$-closed if $\Gamma=\Gamma^\ast$. Subgroups which are $2^\ast$-closed  are in bijection with the set of colored spaces.  These concepts have been investigated in \cite{Wielandt, Siemons82} along with a generalization to regular colorings in \cite{Siemons83}. The combinatorics of $2^\ast$-closed subgroups is very complicated and little is known in general, \cite[p. 1502]{Comb}. In particular, the number of such subgroups is not known, but brute-force search for small $p$ indicates that this number is much less than the number of all subgroups of $\mathfrak{S}_p$ (see Table \ref{tab:tab2}). Even though cyclic subgroups of $\mathfrak{S}_p$ are in general not $2^\ast$-closed, each cyclic group corresponds to a different coloring (see Lemma \ref{lem:cyc}).

For a permutation $\sigma \in \mathfrak{S}_p$, denote its matrix by
\begin{equation}
\label{def:R}
R(\sigma) := \sum_{i=1}^p E_{\sigma(i) i},
\end{equation}
where $E_{ab}$ is the $p\times p$ matrix with $1$ in the $(a,b)$-entry and $0$ in other entries. 
The condition $x_{\sigma(i)\sigma(j)}=x_{ij}$ is then equivalent to $R(\sigma) \cdot x\cdot  \transp{R(\sigma)} = x$. Consequently,
\begin{align}
\label{def:Zsp}
\Zsp_{\Gamma}
= \set{x \in \mathrm{Sym}(p;\real)}
{R(\sigma) \cdot x\cdot \transp{R(\sigma)} = x \mbox{ for all } \sigma \in \Gamma}.
\end{align}

\begin{defi}
	Let $\pi_{\Gamma}\colon \mathrm{Sym}(p;\real)\to \Zsp_{\Gamma}$ be the orthogonal projection
	on  $\Zsp_{\Gamma}$, i.e. 
	the linear map
	such that for any $x\in\mathrm{Sym}(p;\real)$ the element $\pi_{\Gamma}(x)\in \Zsp_{\Gamma}$ is uniquely determined by
	\begin{align}\label{def:proj}
	\Tr{x \cdot y} = \Tr{\pi_{\Gamma}(x) \cdot y}\qquad (y\in \Zsp_{\Gamma}).
	\end{align} 
\end{defi}

In view of \eqref{def:Zsp}, it is clear that 
\begin{align}\label{eq:proj}
\pi_{\Gamma}(x) = \frac{1}{|\Gamma|} \sum_{\sigma \in \Gamma} R(\sigma) \cdot x\cdot \transp{R(\sigma)}
\end{align}
satisfies the above definition. Here $|\Gamma|$ denotes the order of $\Gamma$.
\subsection{$\Zsp_{\Gamma}$ as a Jordan algebra}\label{sec:Jordan}
To derive analytic expression for Gamma-like functions on $\mathcal{P}_\Gamma$ it is convenient to see $\Zsp_\Gamma$ as a Euclidean Jordan algebra and $\mathcal{P}_\Gamma$ as the corresponding symmetric cone. This fact was already observed in \cite{Je88}. We recall here the fundamentals of Jordan algebras, cf. \cite{FaKo94}.
A Euclidean Jordan algebra is a Euclidean space $\A$ (endowed with the scalar product denoted by $\cpling{\cdot}{\cdot}$) equipped with a bilinear mapping (product)
\begin{align*}
\A\times \A \ni \left(x,y\right)\mapsto x\bullet y\in \A
\end{align*}
such that for all $x$, $y$, $z$ in $\A$:
\begin{itemize}
	\item[{\rm (i)}] $x\bullet y=y\bullet x$, 
	\item[{\rm (ii)}] $x\bullet((x\bullet x)\bullet y)=(x\bullet x)\bullet(x\bullet y)$,
	\item[{\rm (iii)}] $\cpling{x}{y\bullet z}=\cpling{x\bullet y}{z}$.
\end{itemize}
A Euclidean Jordan algebra 
is said to be simple if it is not a Cartesian product of two Euclidean Jordan algebras of positive dimensions. We have the following result.
\begin{pro}
	\label{ZJordan}
	The Euclidean space $\Zsp_{\Gamma}$ with  inner product $\cpling{x}{y}=\Tr{x\cdot y}$ and  the Jordan product
	\begin{align}\label{eq:JordanProduct}
	x\bullet y=\tfrac{1}{2}(x\cdot y+y\cdot x),
	\end{align}
	is a Euclidean Jordan algebra. This algebra is generally non-simple.
\end{pro}
\begin{proof}
	Since  $\Zsp_{\Gamma}$ is a subset of the Euclidean Jordan algebra $\mathrm{Sym}(p; \real)$, if it is endowed with Jordan product \eqref{eq:JordanProduct}, 
	conditions {\rm (i)--{(iii)}} are automatically satisfied. Moreover, characterization  \eqref{def:Zsp} of  $\Zsp_\Gamma$ implies that the Jordan product is closed in $\Zsp_\Gamma$, that is, $R(\sigma)\cdot(x\bullet y)=(x\bullet y)\cdot R(\sigma)$ for all $x,y\in\Zsp_{\Gamma}$ and $\sigma\in\Gamma$. The result follows. 
\end{proof} 

Up to linear isomorphism, there are only five kinds of Euclidean simple Jordan algebras. 
Let $\mathbb{K}$ denote the set of either the real numbers $\real$, the complex ones $\complex$ or the quaternions $\mathbb{H}$. Let us write $\mathrm{Herm}(r;\mathbb{K})$ for the space of $r\times r$ Hermitian matrices valued in $\mathbb{K}$. Then $\mathrm{Sym}(r;\real)$, $r\geq 1$, $\mathrm{Herm}(r;\complex)$, $r\geq 2$, $\mathrm{Herm}(r; \mathbb{H})$, $r\geq 2$ are the first three kinds of Euclidean simple Jordan algebras and they are the only ones that will concern us. The determinant and trace in Jordan algebras $\mathrm{Herm}(r;\mathbb{K})$ will be denoted by $\det$ and $\mathrm{tr}$ (see \cite[p. 29]{FaKo94}) respectively, so that they can be easily distinguished from the determinant and trace in $\mathrm{Mat}(n,n;\real)$ which we denote by $\mathrm{Det}$ and $\mathrm{Tr}$.

To each Euclidean  
Jordan algebra $\A$, one can attach the set $\overline{\Omega}$ of Jordan squares, that is, $\overline{\Omega} = \set{x\bullet x}{ x\in\A}$.
The interior $\Omega$ of $\overline{\Omega}$ is a symmetric cone, that is, it is self-dual and homogeneous.
We say that $\Omega$ is irreducible if it is not the Cartesian product of two convex cones. One can prove that an open convex cone is symmetric and irreducible if and only if it is the symmetric cone $\Omega$ of some Euclidean simple Jordan algebra. Each simple Jordan algebra corresponds to a symmetric cone. The  first three kinds
of irreducible symmetric cones
are thus, the symmetric
positive definite real matrices $\mathrm{Sym}^+(r;\real)$ for $r\geq 1$, complex Hermitian positive definite matrices $\mathrm{Herm}^+(r; \complex)$, and quaternionic Hermitian positive definite matrices $\mathrm{Herm}^+(r;\mathbb{H})$, $r\geq 2$.

It follows from Definition \ref{def:Pgamma} and Proposition \ref{ZJordan}  that $\Pcone_{\Gamma}$ is a symmetric cone. In \cite[Proposition III.4.5]{FaKo94} it is stated that any symmetric cone is a direct sum of irreducible symmetric cones. As it will turn out, only three out of the five kinds of irreducible symmetric cones may appear in this decomposition.

Moreover, we will want to represent the elements of the symmetric cones in their real symmetric matrix representations. 
So, we recall that both $\mathrm{Herm}(r; \complex)$ and $\mathrm{Herm}(r;\mathbb{H})$ can be realized as real symmetric matrices, but of bigger dimension. 
For $z=a+b\,i\in\complex$ define 
$M_\complex(z)=\begin{pmatrix}
a & -b \\
b & a
\end{pmatrix}$.
The function $M_\complex$ is a matrix representation of $\mathbb{C}$.
{
	Similarly, any $r\times r$ complex matrix  can be realized as a $(2r)\times(2r)$ real matrix by setting the correspondence
	\[
	\mathrm{Mat}(r,r;\complex)\ni
	\begin{pmatrix}
	z_{i,j}
	\end{pmatrix}_{1\leq i,j\leq r}
	\simeq
	\begin{pmatrix}
	M_\complex(z_{i,j})
	\end{pmatrix}_{1\leq i,j\leq r}\in\mathrm{Mat}(2r,2r;\real),
	\]
	that is, an $(i,j)$-entry of a complex matrix is replaced by its $2\times 2$ real matrix representation. 
	Note that $M_\complex$ maps the space $\mathrm{Herm}(r; \complex)$ of Hermitian matrices into the space $\mathrm{Sym}(2r;\real)$ of symmetric matrices.}
For example, 
\[
{M_\complex \begin{pmatrix}  a & c-d\, i \\ c +d\, i & b \end{pmatrix}
	=}
\begin{pmatrix}
a & 0 & c & d \\
0 & a & -d & c \\
c & -d & b & 0 \\
d & c & 0 & b
\end{pmatrix}.
\]
Moreover, by direct calculation one sees that
\[
\Det{\begin{matrix}
	a & 0 & c & d \\
	0 & a & -d & c \\
	c & -d & b & 0 \\
	d & c & 0 & b
	\end{matrix}} = \det\begin{pmatrix}  a & c-d\, i \\ c +d\, i & b \end{pmatrix}^2.
\]
It can be shown that, in general,
\begin{align}\label{eq:DetTrC}
\Det{M_\complex(Z)}
=
\left[\det\left(Z\right)\right]^2
\quad\mbox{ and }\quad
\Tr{M_{\complex}(Z)} = 2 \,\mathrm{tr}[Z]
\qquad {(Z \in \mathrm{Herm}(r; \complex))}.
\end{align}

Similarly, quaternions can be realized as a $4\times 4$ matrix:
\[
a+b\,i+c\,j+d\,k\simeq \begin{pmatrix}
a+b\, i & -c+d\, i \\
c+d\,i & a-b\, i
\end{pmatrix}
\simeq
\begin{pmatrix}
a & -b & -c & -d \\
b & a & d & -c \\
c & -d & a & b \\
d & c & -b & a 
\end{pmatrix}.
\]
Then, quaternionic $r\times r$ matrices are realized as $(4r)\times(4r)$ real matrices.
Thus, $M_\quat$ maps $\mathrm{Herm}(r; \quat)$ into $\mathrm{Sym}(4r; \real)$. 
Moreover, it is true that
\begin{align}\label{eq:DetTrHerm}
\Det{M_\quat(Z)}
=
\left[\det\left(Z\right)\right]^4
\quad\mbox{ and }\quad
\Tr{M_{\quat}(Z)} = 4 \,\mathrm{tr}[Z]
\qquad {(Z \in \mathrm{Herm}(r; \quat))}.
\end{align}

\subsection{Determining the structure constants and  invariant measure on 	$\Pcone_\Gamma$}\label{sec:25}
As mentioned in the introduction, in order to derive the analytic expression of the Gamma-like functions on $\Pcone_{\Gamma}$, we need the structure constants $(k_i, d_i, r_i)_{i=1}^L$ as well as the invariant measure $\varphi_{\Gamma}$. However, due to Proposition \ref{pro:phi} below, $\varphi_\Gamma(X)$ is expressed in terms of the polynomials $\det(x_i)$, where $x_i\in \mathrm{Herm}(r_i; \mathbb{K}_i)$, $i=1,\ldots,L$, coming from decomposition \eqref{genform}. These can be derived from the decomposition of $\Zsp_{\Gamma}$. Let us note that the constants $(d_i)_i$ and $(k_i)_i$ depend only on the group $\Gamma$, while $r_i$ depend on a particular representation of $\Gamma$, which is $R$. 

In view of decomposition \eqref{genform}, for $X\in \Zsp_\Gamma$,
define $\phi_i(X)= x_i \in \mathrm{Herm}(r_i; \mathbb{K}_i)$
for $i=1, \dots, L$.

\begin{cor}\label{cor:cor11}
	For $X \in \Zsp_\Gamma$, 
	one has
	\begin{equation} \label{eq:detJordan}
	\Det{X} = \prod_{i=1}^L \det (\phi_i(X))^{k_i}.
	\end{equation}
\end{cor}
\begin{proof}
	By \eqref{genform}, we have
	\begin{align*}
	\Det{X} &= \prod_{i=1}^L \Det{ M_{\mathbb{K}_i}(x_i)  \otimes I_{k_i/d_i} }
	= \prod_{i=1}^L \Det{M_{\mathbb{K}_i}(x_i)}^{k_i/d_i} \\
	& 
	= \prod_{i=1}^L \left[ \det (x_i)^{d_i}  \right]^{k_i/d_i}
	= \prod_{i=1}^L \det (x_i)^{k_i},
	\end{align*}
	whence follows the formula.
	We have used \eqref{eq:DetTrC}
	and   \eqref{eq:DetTrHerm}
	for the third equality above.
\end{proof}

\begin{lem}\label{rem:Useful}
	Assume that $\Gamma\subset\mathfrak{S}_p$ and that $(k_i,d_i,r_i)_{i=1}^L$ are the structure constants corresponding to $\Zsp_\Gamma $.  
Assume that we have an irreducible factorization
\begin{equation} \label{eqn:another_factorization}
	\Det{X} = \prod_{j=1}^{L'} f_j(X)^{a_j} \qquad (X \in \Zsp_\Gamma),
\end{equation}
where each $a_j$ is a positive integer, 
each $f_j(X)$ is an irreducible polynomial of $X \in \Zsp_\Gamma$, 
and $f_i \ne f_j$ if $i \ne j$.

Then, $L = L'$, for each $j$ there exists unique $i$ such that $f_j(X)^{a_j} = \det(\phi_i(X))^{k_i}$, and
\begin{enumerate}
\item[\textit a)] $k_i=a_j$,
\item[\textit b)] $r_i$ is the degree of $f_j(X) = \det(\phi_i(X))$,
\item[\textit c)] If $r_i>1$, then $d_i$ can be calculated from 
$r_i + d_i r_i (r_i-1)/2 =  \mathrm{rank}(P_j)$, 
where $P_j$ is the linear operator  defined by $\Zsp_\Gamma\ni x\mapsto P_j(x)=E_j \bullet x \in \Zsp_\Gamma$ and 
$E_j \in \Zsp_\Gamma$ is the gradient of $f_j(X)$ at $X = I_p$.
\end{enumerate}
\end{lem}
\begin{rem}
If $r_i=1$, the determination of
$d_i$ is not needed for writing the block decomposition  of
$\Zsp_\Gamma$, since in this case $\real=
\mathrm{Herm}(1; \real)=\mathrm{Herm}(1; \complex)=
\mathrm{Herm}(1; \mathbb{H})$ and, if $k_i$ is divisible by 2 or by 4, we have
$
M_{\mathbb{K}_i}(x_i)  \otimes I_{k_i/d_i}=x_i I_{k_i}.
$
\end{rem}
\begin{proof}[Proof of Lemma \ref{rem:Useful}]
Since the determinant polynomial of a simple Jordan algebra is always irreducible \cite[Lemma 2.3 (1)]{Upmeier},  
		comparing \eqref{eq:detJordan} and \eqref{eqn:another_factorization}, we obtain $L = L'$, and that, for each $j$, there exists $i$ such that $f_j(X)^{a_j} = \det(\phi_i(X))^{k_i}$.
		From this follows also \textit{a)} and \textit{b)}.
		
Observe that $r_i + d_i r_i (r_i-1)/2 = \dim_\real \,\mathrm{Herm}(r_i; \mathbb{K}_i)$. Point \textit{c)} follows from the fact that $P_j$ coincides with the projection $\bigoplus_{i=1}^L \mathrm{Herm}(r_i; \mathbb{K}_i) 
		\to \mathrm{Herm}(r_i; \mathbb{K}_i)$.  
\end{proof}

The practical significance of the method proposed in this lemma is that neither
representation theory nor group theory is used.
It is a strong advantage when we consider colorings corresponding to a
large number of different groups, for which finding structure constants is very complicated.

\begin{rem}\label{rem:factorization}
	The factorization of multivariate polynomials over an algebraic number field can be done for example in \textsc{Python} (see \textsl{sympy.polys.polytools.factor}) or in \textsc{Mathematica} (see \textsl{Factor}). However, in order to make use of Lemma \ref{rem:Useful}, one has to perform a factorization over the real number field. It turns out that the previously listed tools can be used for this purpose by selecting an appropriate \textsl{Extension} parameter. 
	Indeed, in our setting, the irreducible factorization over the real number field coincides with
	the one over the real cyclotomic field 
	\[
	\mathbb{Q}\left[\zeta+\frac{1}{\zeta}\right]
	=\set{ \sum_{k=0}^{\varphi_E(M)/2-1} q_k \, \left(\zeta + \frac{1}{\zeta}\right)^k}{q_k \in \mathbb{Q},\,\,
		k=0,1, \dots, \varphi_E(M)/2-1},
	\]
	where $\zeta$ is the primitive $M$-th root $e^{2\pi \iu /M}$ of unity 
	with $M$ being the least common multiple of the orders of elements $\sigma \in \Gamma$,
	and 
	$\varphi_E(M)$ is the number of positive integers up to $M$ that are relatively prime to $M$  \cite[Section 12.3]{JPS77}. 
\end{rem}

An example showing the utility of Lemma \ref{rem:Useful} can be found in the Supplementary material \cite[Section 3.1]{GIKM_SM}.

\subsection{Finding structure constants and construction of the orthogonal matrix  $U_\Gamma$ when $\Gamma$ is cyclic}
\label{cyclic:basic}
We now show that, 
when the group $\Gamma$ is generated by one permutation 
$\sigma \in \mathfrak{S}_p$, 
the orthogonal matrix $U_\Gamma$ can be constructed explicitly,
and we obtain the structure constants $r_i, \, k_i$ and $d_i$ easily.

Let us consider the $\Gamma$-orbits in $\{1,2, \dots, p\}$. 
Let $\{i_1, \dots, i_C\}$ be a complete system of representatives of the $\Gamma$-orbits,
and for each $c=1, \dots, C$, 
let $p_c$ be the cardinality of the $\Gamma$-orbit through $i_c$.
The order $N$ of $\Gamma$ equals the least common multiple of $p_1,\,p_2, \dots, p_C$ and one has $\Gamma = \{\mathrm{id},\, \sigma,\, \sigma^2, \ldots, \sigma^{N-1}\}$. In what follows, we treat $0$ as a multiple of $N$. 

\begin{thm}\label{thm:cyclic1}
	Let $\Gamma=\left<\sigma\right>$ be a cyclic group of order $N$. 
	For $\alpha=0,1,\ldots,\lfloor\frac{N}{2}\rfloor$ set
	\begin{align*}	
	r_\alpha^\ast &= \#\set{c\in\{1,\ldots,C\}}{\alpha\, p_c  \mbox{ is a multiple of }N}\\
	d_\alpha^\ast &= \begin{cases} 1 & (\alpha = 0 \mbox{ or }N/2) \\ 2 & \mbox{(otherwise)}. \end{cases}
	\end{align*}
	Then we have $L=\#\set{\alpha}{r_\alpha^\ast>0}$,
	$r = (r_\alpha^\ast;\, r_\alpha^\ast>0)$ 
	and
	$k = d = (d_\alpha^\ast;\, r_\alpha^\ast>0)$.
\end{thm}
Note that, $r_0^\ast$ equals the number $C$ of cycles in a decomposition of a permutation.

\begin{exa}\label{ex:splitted1}
	Let us consider 
	$\sigma = (1,2,3){(4,5)}{(6)} 
	\in \mathfrak{S}_6$. 
	The three $\Gamma$-orbits are $\{1,2,3\}, \{4,5\}$ and $\{6\}$.
	Set $i_1=1$, $i_2=4$, $i_3=6$.
	Then $p_1 = 3$, $p_2 = 2$, $p_3 = 1$. 
	We have $N=6$.
	We count $r_0^\ast = 3$, $r_1^\ast = 0$, $r_2^\ast = 1$, $r_3^\ast = 1$, 
	so that $r=(3,1,1)$. Since $d=(1,2,1)$, we have $\Zsp_\Gamma \simeq \mathrm{Sym}(3; \real) \oplus \mathrm{Herm}(1; \complex) \oplus \mathrm{Sym}(1; \real)$.
\end{exa}

For $c=1,\ldots,C$ define $v^{(c)}_1, \dots, v^{(c)}_{p_c} \in \real^p$ by
\begin{align*}
v^{(c)}_1 &:= \sqrt{\frac{1}{p_c}} \sum_{k=0}^{p_c-1} e_{\sigma^k(i_c)}, \\
v^{(c)}_{2\beta} 
&:= \sqrt{\frac{2}{p_c}} \sum_{k=0}^{p_c-1} \cos \Bigl( \frac{2\pi \beta k}{p_c} \Bigr) e_{\sigma^k(i_c)} \qquad (1 \le \beta < p_c/2),\\
v^{(c)}_{2\beta+1} 
&:= \sqrt{\frac{2}{p_c}} \sum_{k=0}^{p_c-1} \sin \Bigl( \frac{2\pi \beta k}{p_c} \Bigr) e_{\sigma^k(i_c)}
\qquad (1 \le \beta < p_c/2),\\
v^{(c)}_{p_c} &:= \sqrt{\frac{1}{p_c}} \sum_{k=0}^{p_c-1} \cos (\pi k) e_{\sigma^k(i_c)}
\qquad\quad\,\,\, (\mbox{if }p_c \mbox{ is even}).
\end{align*}

\begin{thm}\label{thm:cyclic2}
	The orthogonal matrix $U_\Gamma$ from Theorem \ref{thm:intro} can be obtained by arranging column vectors $\{v^{(c)}_k\}$, $1 \le c \le C$, $1\le k \le p_c$ in the following way:
	we put $v^{(c)}_k$ earlier than $v^{(c')}_{k'}$
	if\\ 
	{\rm (i)} $\frac{[k/2]}{p_c} < \frac{[k'/2]}{p_{c'}}$, or\\
	{\rm (ii)} $\frac{[k/2]}{p_c} = \frac{[k'/2]}{p_{c'}}$ 
	and $c<c'$, or\\
	{\rm (iii)} $\frac{[k/2]}{p_c} = \frac{[k'/2]}{p_{c'}}$ and $c = c'$ and $k$ is even and $k'$ is odd.
\end{thm}
Proofs of the above results are presented in the Supplementary material.
We shall see there that $R(\sigma)$ acts on the 2-dimensional space 
spanned by $v^{(c)}_{2\beta}$ and $v^{(c)}_{2\beta + 1}$
as a rotation with the angle $2\pi \beta/p_c$, $1\leq \beta<p_c/2$.
The condition (i) means that the angle for $v^{(c)}_k$ is smaller than 
the one for $v^{(c')}_{k'}$.

\begin{exa}\label{ex:splitted2}
We continue Example \ref{ex:splitted1}.
	According to Theorem  \ref{thm:cyclic2},
	\[
	U_\Gamma = \begin{pmatrix} v^{(1)}_1,  v^{(2)}_1,  v^{(3)}_1,  v^{(1)}_2,  v^{(1)}_3,  v^{(2)}_2 \end{pmatrix}
	= \begin{pmatrix} 
	1/\sqrt{3} & 0 & 0 & \sqrt{2/3} & 0 & 0 \\
	1 / \sqrt{3} & 0 & 0 & -\sqrt{1/6} & 1/\sqrt{2} & 0\\
	1 / \sqrt{3} & 0 & 0 & -\sqrt{1/6} & -1/\sqrt{2} & 0\\
	0 & 1/\sqrt{2} & 0 & 0 & 0 & 1/\sqrt{2} \\ 
	0 & 1/\sqrt{2} & 0 & 0 & 0 & -1/\sqrt{2} \\ 
	0 & 0& 1 & 0 & 0& 0 \end{pmatrix}.
	\]
	Then we have
	\[ 
	\transp{U_\Gamma} \cdot R(\sigma^k) \cdot U_\Gamma
	= \begin{pmatrix} I_3 \otimes B_0(\sigma^k)& & \\ & B_2(\sigma^k) & \\ & & B_3(\sigma^k) \end{pmatrix},
	\]
	where $B_0(\sigma^k)=1$,  $B_2(\sigma^k)=\mathrm{Rot}(\frac{2\pi  k}{3}) \in \mathrm{GL}(2; \real)$ and $B_3(\sigma^k)=(-1)^k$. Here 
	$\mathrm{Rot}(\theta)$
	denotes the rotation matrix  
	$\begin{pmatrix} \cos \theta & -\sin \theta \\ \sin \theta & \cos \theta \end{pmatrix}$
	for $\theta \in \real$.\\
	The block diagonal decomposition of $\Zsp_\Gamma$ is 
	\[
	\transp{U_\Gamma} \cdot \Zsp_\Gamma \cdot U_\Gamma
	= \set{ \begin{pmatrix} x_1 &  & \\ & x_2 I_2 & \\ & & x_3 \end{pmatrix}}{x_1 \in \mathrm{Sym}(3;\real),\,\,x_2,x_3 \in \real}. 
	\]
\end{exa}

\begin{rem}
	In the cyclic case we have $k=d$ and so the formula \eqref{genform} holds without the Kronecker product terms. Since $d_i\in\{1,2\}$, the quaternionic case never occurs.
\end{rem}

\section{Gamma integrals and normalizing constants}\label{gammaintegral}
\subsection{Gamma integrals on irreducible symmetric cones}
Let $\Omega$ be one of the first three kinds of irreducible symmetric cones, that is, $\Omega=\mathrm{Herm}^+(r; \K)$, where $\K\in\{\real,\complex,\mathbb{H}\}$.
As before, determinant and trace on corresponding Euclidean Jordan algebras are denoted by $\det$ and $\mathrm{tr}$.
Then, we have the relation
\[
\dim\Omega = r+\frac{r(r-1)}{2}d,
\]
where $d=1$ if $\K=\real$, $d=2$ if $\K=\complex$ and $d=4$ if $\K=\mathbb{H}$.

Recall that Euclidean measure is the volume measure induced by the Euclidean metric.
Let $m(\mathrm{d}x)$ denote the Euclidean measure associated with the Euclidean structure defined on $\mathcal{A} = \mathrm{Herm}(r;\K)$ by $\cpling{x}{y}=\mathrm{tr}[x\bullet y] = \mathrm{tr}[x\cdot y]$. The Gamma integral
\[
\Gamma_{\Omega}(\lambda):=\int_\Omega \det(x)^{\lambda} e^{-\mathrm{tr}[x]} \det(x)^{-\dim\Omega/r} m(\mathrm{d}x)
\]
is finite if and only if $\lambda> \frac12(r-1)d=\dim\Omega/r-1$ and in such case
\begin{align}\label{eq:GammaSymmetric}
\Gamma_{\Omega}(\lambda)= 
(2\pi)^{(\dim\Omega-r)/2}\Gamma(\lambda)\Gamma(\lambda-d/2)\ldots\Gamma(\lambda-(r-1)d/2).
\end{align}
Moreover, one has
\begin{align} \label{eq:GammaSymmetricY}
\int_{\Omega} \det(x)^{\lambda} e^{- \mathrm{tr}[x\cdot y]}\det(x)^{-\dim\Omega/r}m(\mathrm{d}x) = \Gamma_{\Omega}(\lambda) \det(y)^{-\lambda}
\end{align}
for any $y\in\Omega$.

The measure $\mu_{\Omega}(\mathrm{d}x)=\det(x)^{-\dim\Omega/r} m(\mathrm{d}x)$ is  invariant in the following sense. Let $\mathrm{G}(\Omega)$ be the linear automorphism group of $\Omega$, that is, the set $\set{g\in \mathrm{GL}(\A)}{g\,\Omega=\Omega}$, where $\A$ is the associated Euclidean Jordan algebra. Then, the measure $\mu_{\Omega}$ is a $\mathrm{G}(\Omega)$-invariant measure in the sense that for any Borel measurable set $B$ one has
\[
\mu_{\Omega}(g^{-1} B) = \mu_{\Omega}(B) \qquad (g\in\mathrm{G}(\Omega)).
\]

\subsection{Gamma integrals on the cone $\Pcone_\Gamma$}

We endow the space $\Zsp_{\Gamma}$ with the scalar product 
\[
\cpling{x}{y}=\Tr{x\cdot y}\qquad (x,y\in \Zsp_{\Gamma}).
\]
Let $\dd X$ denote the Euclidean measure on the Euclidean space
$(\Zsp_{\Gamma}, \cpling{\cdot}{\cdot})$. 
Let us note that this normalization is not important in the Bayesian model selection procedure as there we always consider quotients of integrals.

\begin{exa} Consider $p=3$ and $\Gamma=\mathfrak{S}_3$. The space $\Zsp_\Gamma$ is $2$-dimensional and it consists of matrices of the form (see Example \ref{ex1})
	\[
	X=\begin{pmatrix}
	a &b &b\\b&a&b\\ b&b&a\end{pmatrix}
	\]
	for $a,b\in\real$.
	Since $\| X \|^2 = \Tr{X^2} =  3a^2 + 6b^2 = \transp{v} v$ with $\transp{v} = (\sqrt{3}a, \sqrt{6}b)$, we have $\dd X=\sqrt{3}\sqrt{6}\dd a\dd b=3\sqrt{2}\dd a\dd b$.
\end{exa}
Generally,  
if $m_i$ denotes the Euclidean measure on 
$\mathcal{A}_i := \mathrm{Herm}(r_i; \mathbb{K}_i)$
with the inner product defined from the Jordan algebra trace
(recall \eqref{eq:DetTrC} and \eqref{eq:DetTrHerm}),
then \eqref{genform} implies that
for $X \in \Zsp_{\Gamma}$ we have
\[ 
\| X\|^2 = \cpling{X}{X}= \sum_{i=1}^L \frac{k_i}{d_i} \Tr{ M_{\mathbb{K}_i}(x_i)^2 }
= \sum_{i=1}^L k_i \mathrm{tr}[ (x_i)^2],
\]
which implies that 
\begin{equation} \label{eq:measJordan}
\dd X = \prod_{i=1}^L \left( \sqrt{k_i} \right)^{\dim \Omega_i} m_i (\mathrm{d}x_i) =  e^{B_\Gamma}\prod_{i=1}^L m_i(\mathrm{d}x_i),
\end{equation}
where
	\begin{align}\label{eq:B}
	B_\Gamma := \frac12\sum_{i=1}^L (\dim\Omega_i) (\log k_i).
\end{align}

\begin{defi}\label{def:phi}
	Let $\mathrm{G}(\Pcone_\Gamma)=\set{g \in \mathrm{GL}(p; \real)}{ g\Pcone_\Gamma = \Pcone_\Gamma}$ be the linear automorphism group of $\Pcone_\Gamma$. We define
	the $\mathrm{G}(\Pcone_\Gamma)$-invariant measure $\varphi_\Gamma(X)\dd X$ by 
	\[
	\varphi_\Gamma(X) = e^{B_\Gamma}
 	{ \left( \prod_{i=1}^L \frac{1}{ \Gamma_{\Omega_i}(\dim \Omega_i /r_i)}\right) }
	\int_{\Pcone_{\Gamma}^\ast} e^{-\Tr{X\cdot Z}}\dd Z,
	\]
	where $\Pcone_{\Gamma}^\ast = \set{y\in \Zsp_{\Gamma}}{\Tr{y\cdot x}>0, \forall\,{x\in\overline{\mathcal{P}_{\Gamma}}}\setminus\{0\}}$ is the dual cone of $\mathcal{P}_\Gamma$.
\end{defi}

\begin{pro}\label{pro:phi}
	We have
	\begin{align}\label{eq:phi}
	\varphi_\Gamma(X) = \prod_{i=1}^L \det(\phi_i(X))^{- \dim\Omega_i/r_i}.
	\end{align}
\end{pro}

The proofs of the following results of this section can be found in the supplementary material \cite{GIKM_SM}.

\begin{defi}
	The Gamma function of $\Pcone_\Gamma$ is defined by the following integral
	\begin{equation}\label{def:gamma}
	\Gamma_{\Pcone_\Gamma}(\lambda):=\int_{\Pcone_\Gamma} \Det{X}^\lambda e^{-\Tr{X}} \varphi_\Gamma(X)\dd X,
	\end{equation}
	whenever it converges.
\end{defi}

\begin{thm}\label{thm:IntegralPhi}
	The integral \eqref{def:gamma} converges if and only if 
	\begin{align}
	\label{eq:convCondition1}
	\lambda > \max_{i=1,\ldots,L} \left\{ \frac{(r_i-1)d_i}{2 k_i}\right\}
	\end{align}
	and, for these values of $\lambda$, we have
	\begin{align}\label{eq:GammaPCone}
	\Gamma_{\Pcone_\Gamma}(\lambda) = e^{-A_\Gamma \lambda+B_\Gamma} \prod_{i=1}^L \Gamma_{\Omega_i}(k_i \lambda)
	\end{align}
	where $\Gamma_{\Omega_i}$ is given in \eqref{eq:GammaSymmetric}, $B_\Gamma$ in \eqref{eq:B} and
	\begin{align}\label{eq:A}
	A_\Gamma := \sum_{i=1}^L r_i\, k_i\log k_i.
	\end{align}
	
	Moreover, if $Y\in\Pcone_\Gamma$ and \eqref{eq:convCondition1} holds true, then
	\begin{align}\label{eq:IntegralPhi}
	\int_{\Pcone_\Gamma} \Det{X}^\lambda e^{-\Tr{Y\cdot X}}  \varphi_\Gamma(X)\dd X = \Gamma_{\Pcone_\Gamma}(\lambda) \Det{Y}^{-\lambda}.
	\end{align}
\end{thm}


We also have the following result
\begin{thm}\label{thm:IntegralNoPhi}
	If $Y\in\Pcone_\Gamma$ and
	\begin{align*}
	\lambda > \max_{i=1,\ldots,L} \left\{-\frac{1}{k_i}\right\}, 
	\end{align*}
	then  
	\begin{align} \label{eq:IntegralNoPhi}
	\int_{\Pcone_\Gamma} \Det{X}^\lambda e^{-\Tr{Y\cdot X}} \dd X
	= e^{- A_\Gamma \lambda - B_\Gamma} 
	\prod_{i=1}^L \Gamma_{\Omega_i  }\left( k_i\,\lambda+ \frac{\dim\Omega_i}{r_i} \right) \frac{\varphi_\Gamma(Y)}{
	\Det{Y}^{\lambda}}. 
	\end{align}
\end{thm}

\subsection{RCOP-Wishart laws on $\Pcone_{\Gamma}$}
Let $\Sigma \in \Pcone_{\Gamma}\subset \mathrm{Sym}^+(p;\real)$ and consider i.i.d.
random vectors $Z^{(1)}, \dots, Z^{(n)}$ following the $\mathrm{N}_p(0, \Sigma)$ distribution. Define $U_i=Z^{(i)} \cdot \transp{\mbox{$Z^{(i)}$}}$, $i=1,\ldots,n$, and $U=\sum_{i=1}^n U_i$.
We note that such model is clearly not invariant
under changing the scale of variables: random vector $\mathrm{diag}(\underline{\alpha})\cdot Z^{(1)}$ for $\underline{\alpha}\in\real ^p$  is in general not invariant under any permutation subgroup. Such issue is an immanent property of RCON models (a generalization of RCOP models) and was noticed already in \cite{HL08}. The authors recommend to  keep all variables in the same units.

Our aim is to analyze the probability distribution of the random matrix
\[
W_n =  \pi_\Gamma(U)=\pi_{\Gamma}(U_1  + \dots + U_n) = \pi_{\Gamma}(U_1)+\ldots+\pi_{\Gamma}(U_n).
\]

In the rest of this section,  
we find $n_0$ such that for $n\geq n_0$ the random matrix $W_n$ follows an absolutely continuous law, and we compute its density. Further, we extend the shape parameter
to a continuous range and define the RCOP-Wishart law on $\Pcone_\Gamma$.

We start with the following easy result.
\begin{lem}\label{lem:Laplace}
	For any $\theta\in\mathrm{Sym}^+(p;\real)$ we have
	\[
	\mathbb{E} e^{ -\Tr{\theta \cdot \pi_{\Gamma}\left(U_1\right) }} = \Det{I_p+2\, \Sigma\cdot\pi_{\Gamma}(\theta)}^{-1/2}.
	\]
\end{lem}
\begin{proof}
	Using \eqref{def:proj} repeatedly we have
	\[	
	\Tr{ \theta \cdot \pi_{\Gamma}\left(U_1\right) } = 
	\Tr{ \pi_{\Gamma}(\theta) \cdot \pi_{\Gamma}\left(U_1\right) } = 
	\Tr{ \pi_{\Gamma}(\theta) \cdot U_1}.
	\]
	The assertion follows from the usual multivariate Gauss integral.
\end{proof}

\begin{pro}\label{pro:MLE}
	The law of $W_n$ is absolutely continuous on  $\Pcone_{\Gamma}$ if and only if 
	\begin{align}\label{eq:n0}
	n \ge n_0 := \max_{i=1,\ldots,L}\left\{\frac{r_i d_i}{k_i}\right\}. 
	\end{align}
	If $n \geq n_0$, then its density function with respect to $\dd X$ is given by 
	\begin{align}\label{eq:density}
	\frac{\Det{X}^{n/2} e^{- \tfrac12\Tr{X\cdot \Sigma^{-1}}} }
	{\Det{2\Sigma}^{n/2} \Gamma_{\Pcone_\Gamma}(\frac{n}{2})} \varphi_\Gamma(X) {\bf 1}_{\Pcone_{\Gamma}}(X).
	\end{align}
\end{pro}
\begin{proof}
	With $\lambda=n/2$, condition \eqref{eq:convCondition1} becomes
	\[
	n > \max_{i=1,\ldots,L}\left\{\frac{(r_i-1) d_i}{k_i}\right\}. 
	\]
	Since the quotient $k_i/d_i$ is an integer, the last condition  is equivalent to \eqref{eq:n0}. 
	
	In view of Lemma \ref{lem:Laplace}, it is enough to show that $W_n$ has density \eqref{eq:density} if and only if for any $\theta\in\Pcone_{\Gamma}$,
	\begin{align}\label{Eq:LaplaceW}
	\E e^{-\Tr{\theta\cdot W_n}} = \Det{I_p+2\,\Sigma\cdot \theta}^{-n/2}.
	\end{align}
	This follows directly from \eqref{eq:IntegralPhi}.
\end{proof}
It is known that the MLE exists and is unique if and only if the sufficient
statistic lies in the interior of its convex support, see \cite{BN14}. It is clear that if \eqref{eq:n0} is not satisfied, then the support of $W_n$ is contained in the boundary of $\Pcone_\Gamma$. Recall that the orthogonal projection $\pi_\Gamma$ is given by \eqref{eq:proj}.

\begin{cor}\label{cor:WisMLE}
	The MLE of $\Sigma$ exists if and only if  the number of samples $n$ satisfies \eqref{eq:n0}. If it exists, it is given by
	\[
	\hat{\Sigma} = \frac{1}{n} \pi_{\Gamma}(U_1 + \dots + U_n).
	\]
\end{cor}
The above result has been already proven in \cite[Theorem 5.9]{A75} (see also \cite[Sec. A.3, A.4]{AM98}). 

\begin{rem}
	If $U = (U_1+ \dots + U_n)/n$ is positive definite, then $U$ can be regarded as an empirical covariance matrix. 	The Kullback-Leibler divergence between $\mathrm{N}_p(0, U)$ and $\mathrm{N}_p(0, \Sigma)$ is equal to 
	$\frac{1}{2} \Bigl\{ \log \det \Sigma + \mathrm{tr}\, U \Sigma^{-1} 
	- \log \det(U) - p \Bigr\}$, 
	which is obviously minimized by the MLE $\hat{\Sigma}$. 
	Therefore, Corollary \ref{cor:WisMLE} implies that $\pi_\Gamma(U)$ is the Kullback-Leibler  projection of $U$ onto $\Pcone_\Gamma$, \cite{GC98}. We note that the KL projection in general is not linear, whereas $\pi_\Gamma$ clearly is.
\end{rem}

Let us recall that the MLE of $\Sigma$ in the standard normal model exists if and only if 
$n \ge p$. We recover this case for $\Gamma=\{\mathrm{id}\}$, since then we have $L=1$, $r_1=p$ and $k_1=d_1=1$. 

When $n < n_0$, the law of $W_n$ is singular,
and it can be described as a direct product of 
the singular Wishart laws on the irreducible symmetric cones $\Omega_i$, see e.g. \cite{HaLa01}.

\begin{defi}\label{RCOP-Wishart ETA}
	Let $\eta>\max\set{ (r_i-1)\frac{d_i}{k_i}}{i=1,\ldots,L}$ and $\Sigma\in\Pcone_{\Gamma}$. 
	The RCOP--Wishart law $W^\Gamma_{\eta, \Sigma}$ is defined by its density
	\begin{align}\label{eq:density ETA}
	W^\Gamma_{\eta, \Sigma}(\mathrm{d}X) = 
	\frac{\Det{X}^{\eta/2} e^{- \tfrac12\Tr{X\cdot \Sigma^{-1}}} }
	{\Det{2\Sigma}^{\eta/2} \Gamma_{\Pcone_\Gamma}(\frac{\eta}{2})} \varphi_\Gamma(X) {\bf 1}_{\Pcone_{\Gamma}}(X)\dd X.
	\end{align}
\end{defi}
With this new notation, we see that if \eqref{eq:n0} is satisfied, then $W_n\sim W_{n,\Sigma}^\Gamma$.

\begin{lem}\label{lem:inv}
	The Jacobian of the transformation 
	\[\Pcone_{\Gamma}\ni X\mapsto X^{-1}\in \Pcone_{\Gamma}\]
	equals $\varphi_\Gamma(X^{-1})^2$.
\end{lem}
Proof of the lemma can be found in the Supplementary material. By this lemma we obtain another useful formula for the invariant measure, namely
\[
\varphi_{\Gamma}(X)=\mathrm{Det}_{\mathrm{End}}( \mathbb{P}_{X}) ^{-1/2}\qquad(X\in\mathcal{Z}_\Gamma),
\] where $\mathrm{Det}_{\mathrm{End}}$ is the determinant in the space of endomorphisms of $\Zsp_{\Gamma}$ and for any $X\in \Zsp_{\Gamma}$ by $\mathbb{P}_{X}$  we denote the linear map on $\Zsp_{\Gamma}$ to itself defined by $\mathbb{P}_{X} Y = X\cdot Y\cdot X$. 
  Lemma \ref{lem:inv} gives also the following result.
\begin{pro}
	Let $W\sim W_{\eta,\Sigma}^\Gamma$ with $\eta>\max\set{ (r_i-1)d_i/k_i}{i=1,\ldots,L}$ and $\Sigma\in\Pcone_{\Gamma}$. Then its inverse $Y=W^{-1}$ has density
	\[
	\frac{\Det{Y}^{-\eta/2} e^{- \tfrac12\Tr{Y^{-1}\cdot \Sigma^{-1}}} }
	{\Det{2\Sigma}^{\eta/2} \Gamma_{\Pcone_\Gamma}(\frac{\eta}{2})} \varphi_\Gamma(Y) {\bf 1}_{\Pcone_{\Gamma}}(Y).		
	\]
\end{pro}

\subsection{The Diaconis-Ylvisaker conjugate prior for $K$}\label{sec:DY}
The Diaconis-Ylvisaker conjugate prior (\cite{DY79}) for the canonical parameter $K=\Sigma^{-1}$ is given by
\[
f(K;\delta,D) = \frac1{I_\Gamma (\delta,D)} {\Det{K}^{(\delta-2)/2} e^{- \tfrac12 \Tr{K \cdot D}} }
{\bf 1}_{\Pcone_{\Gamma}}(K),
\]
for hyper-parameters $\delta>2\max\set{1-1/k_i}{i=1,\ldots,L}$ and $D\in \Pcone_{\Gamma}$.  By \eqref{eq:IntegralNoPhi}, the normalizing constant is equal to
\begin{equation}\label{eq:I YD}
I_\Gamma(\delta,D)= e^{- A_\Gamma (\delta-2)/2 - B_\Gamma} 
\prod_{i=1}^L \Gamma_{\Omega_i  }\left( k_i\, \frac{\delta-2}{2}+ \frac{\dim\Omega_i}{r_i} \right) \frac{\varphi_\Gamma\left(\tfrac12D\right)}{
\Det{\tfrac12D }^{(\delta-2)/{2}}},
\end{equation}
where $A_\Gamma$, $B_\Gamma$ and $\varphi_\Gamma$ are given in \eqref{eq:A}, \eqref{eq:B} and \eqref{eq:phi}.

We note that despite the fact that the choice of hyper-parameters is not scale invariant, statisticians usually take $\delta=3$ and $D=I_p$, see e.g. \cite{Hel18b}.

\section{Model selection}\label{modelselection}

Bayesian model selection on all colored spaces seems at the moment intractable. This is due in great part to a poor combinatorial description of the colored spaces $\mathcal{Z}_\Gamma$. In particular, the number of such spaces, that is, $\#\set{\mathcal{Z}_\Gamma}{\Gamma\in\mathfrak{S}_p}$ is generally unknown for large $p$. It was shown in \cite{Ge11} that these colorings constitute a lattice with respect to the usual inclusion of subspaces. However the structure of this lattice is rather complicated and is unobtainable for big $p$. This, in turn, does not allow to define a Markov chain  with known transition probabilities on such colorings. Finally, the fundamental problem
which prevents us from doing Bayesian model selection on all colored spaces for arbitrary $p$ is the following. In order to compute Bayes factors, one has to be able to find the structure constants $(k_i,d_i,r_i)_{i=1}^L$ for arbitrary subgroups of $\mathfrak{S}_p$. This is equivalent to finding irreducible representations over reals for an arbitrary finite group, 
which is very hard in general, although general algorithms have been developed for this issue
(see \cite{PS96}).

In this section, we are making a step forward in the problem of model selection for colored models in two ways. In Section \ref{selection:cyclic}, we use the results of Section  \ref{cyclic:basic}, to obtain the structure constants when we restrict our search to the space of colored models generated by a cyclic group, that is, when $\Gamma=\left<\sigma\right>$ for $\sigma\in \mathfrak{S}_p$ and we propose a model selection procedure restricted to the cyclic colorings. In Section \ref{selection:general}, we use Lemma \ref{rem:Useful} and Remark \ref{rem:factorization} to obtain  the irreducible representations of $\Zsp_{\Gamma}$ and the structure constants
by factorization of the determinant. We apply this technique to do model selection for the four-dimensional example given by Frets' data since, in that case, there are only $22$ models and we can compute all the Bayes factors. 

\subsection{Model selection within cyclic groups}
\label{selection:cyclic}
The smaller space of cyclic colorings  has a much better combinatorial description.
In particular, the following result can be proved.
\begin{lem}\label{lem:cyc}
	If $\Zsp_{\left<\sigma\right>}=\Zsp_{\left<\sigma'\right>}$ for some $\sigma,\sigma'\in\mathfrak{S}_p$, then $\left<\sigma\right>=\left<\sigma'\right>$.
\end{lem} 
This result allows us to calculate the number of different colorings corresponding to cyclic groups, that is, the number of labeled cyclic subgroups of the symmetric group $\mathfrak{S}_p$, which can be found in OEIS, sequence A051625 (see the last column of Table \ref{tab:tab2}). 

\begin{table}
	\caption{Number of all subgroups of a symmetric group, number of their conjugacy classes, number of different colorings and a number of cyclic groups}
	\label{tab:tab2} 
		\begin{tabular}{@{}lrrrr@{}}
			\hline
			$p$ & $\#\mbox{subgroups of }\mathfrak{S}_p$ & $\#\mbox{conjugacy classes of }\mathfrak{S}_p$ & $\#\mbox{different }\Zsp_{\Gamma}$ & $\#\mbox{cyclic groups}$\\
			\hline 
			1 & 1 & 1  & 1 & 1\\ 
			2 & 2 & 2 & 2 & 2\\ 
			3 & 6 & 4 & 5 & 5\\ 
			4 & 30 & 11 & 22 & 17 \\ 
			5 & 156 & 19 & 93 & 67\\ 
			6 & 1\,455 & 56 & 739 & 362 \\ 
			7 & 11\,300 & 96 & 4\,508 & 2039\\
			8 & 151\,221 & 296 & ? & 14\,170 \\
			9 & 1\,694\,723 & 554 & ? & 109\,694 \\
			10 &  29\,594\,446 & 1\,593 &  ? & 976\,412\\
			18 & $\approx 7.6\cdot 10^{18}$ & $7.3\cdot 10^6$ & ? & $\approx 7.1\cdot 10^{14}$\\
			\hline
	\end{tabular}	
\end{table}

We will present two applications of the Metropolis-Hastings algorithm. In the first one, the Markov chain will move on the space of cyclic groups. The drawback of this first approach is that we need to compute  the proposal distribution $g$, whose computational complexity grows faster than quadratically as $p$ increases (see \eqref{eq:defg}). In the second algorithm, we consider a larger state space
$\mathfrak{S}_p$, which allows us to consider an easy proposal distribution. However, this comes at the cost of slower convergence of the posterior probabilities (see Theorem \ref{thm:Alg2}).

\subsubsection{First approach}
Each cyclic subgroup $\Gamma$ can be uniquely represented by a permutation, which is minimal in the lexicographic order within permutations generating $\Gamma$. Let $\nu(\Gamma)\in\mathfrak{S}_p$ be such a permutation, that is,
\[
\nu(\Gamma) = \min\set{\sigma\in\mathfrak{S}_p}{\left<\sigma\right>=\Gamma}.
\]
Define 
\begin{align}\label{eq:defct}
c_t:=\left< \nu(c_{t-1})\circ x_t  \right>,
\end{align}
where $c_0$ is a fixed cyclic subgroup and 
$(x_t)_{t\in\mathbb{N}}$ is a sequence of i.i.d. random transpositions distributed uniformly, that is, $\mathbb{P}(x_t=\alpha)=1/\binom{p}{2}$ for any $\alpha\in\mathcal{T}:=\{(i,\,j)\in\mathfrak{S}_p\}$.
Clearly, the sequence $(c_t)_t$ is a Markov chain. Its state space is the set of all cyclic subgroups of $\mathfrak{S}_p$. Moreover, the trivial subgroup $\{\mathrm{id}\}$ can be reached from any subgroup $c_t$ (and vice versa) in a finite number of
steps with positive probability. Thus the chain $(c_t)_t$ is irreducible.
The proposal distribution in the Metropolis-Hastings algorithm is the conditional distribution of $c_t|c_{t-1}$. It is proportional to the number of possible transitions from $c$ to $c^\prime$, that is,
\begin{align}\label{eq:defg}
g\left(c^\prime|c\right) := \frac{\#\set{(i,\, j)\in\mathfrak{S}_p}{ c^\prime = \left<\nu(c)\circ(i,\,j)\right>}}{ \binom{p}{2}},
\end{align}
where $c$ and $c^\prime$ are cyclic subgroups.

We follow the principles of Bayesian model selection for graphical models, presented, for example, in \cite[Chapter 10, p.247]{NewBook}.
Let $\Gamma$ be uniformly distributed on the set $\mathcal{C}:=\set{\left<\sigma\right>}{\sigma\in\mathfrak{S}_p}$ of cyclic subgroups of $\mathfrak{S}_p$. We assume that $K|\{\Gamma=c\}$, $c\in\mathcal{C}$, follows the Diaconis-Ylvisaker conjugate prior distribution on $\Pcone_{c}$ with hyper-parameters $\delta$ and $D$, that is,
\[
f_{K|\Gamma=c}(k)=\frac1{I_{c} (\delta,D)} {\Det{k}^{(\delta-2)/2} e^{- \tfrac12 \Tr{D\cdot k}} }
{\bf 1}_{\Pcone_{c}}(k),
\]
where the normalizing constant is given in \eqref{eq:I YD}.
Suppose that $Z_1,\ldots, Z_n$ given $\{K=k, \Gamma=c\}$ are i.i.d. $\mathrm{N}_p(0,k^{-1})$ random vectors with $k\in \Pcone_{c}$.
Then, it is easily seen that we have
\begin{align}\label{eq:defpp}
\mathbb{P}\left(\Gamma=c|Z_1,\ldots,Z_n\right) \propto \frac{ I_{c}(\delta + n, D+U)}{I_{c}(\delta,D)}\qquad(c\in\mathcal{C})
\end{align}
with $U=\sum_{i=1}^n Z_i\cdot \transp{Z_i}$. These derivations allow us to run the Metropolis-Hastings algorithm restricted to cyclic groups, as follows.

\begin{alg}
	Starting from a cyclic group $C_0\in\mathcal{C}$, repeat the following two steps for $t=1, 2,\ldots$:
	\begin{enumerate}
		\item Sample $x_t$ uniformly from the set $\mathcal{T}$ of all transpositions and set $c^\prime = \left<\nu(C_{t-1})\circ x_t\right>$;
		\item Accept the move $C_t = c^\prime$ with probability
		\[
		\min\left\{ 1, \frac{I_{c^\prime}(\delta+n,D+U)\,\, I_{C_{t-1}}(\delta,D) }{I_{c^\prime}(\delta,D)\,\, I_{C_{t-1}}(\delta+n,D+U)  }\,\, \frac{g\left(C_{t-1}|c^\prime\right) }{g\left(c^\prime|C_{t-1}\right) }
		\right\}
		\]
		If the move is rejected, set $C_t=C_{t-1}$.
	\end{enumerate}
\end{alg}

\subsubsection{Second approach}\label{sec:sa}
It is known that $\left<\sigma\right>=\left<\sigma'\right>$ if and only if $\sigma'=\sigma^k$ for some $k\in \beta(|\sigma|)$, where
\begin{align}\label{eq:defbeta}
\beta(n) = \set{k\in\{1,\ldots,n\}}{k\mbox{ and }n\mbox{ are relatively prime}}
\end{align}
and $|\sigma|$ denotes the order of $\sigma$.
Let $\mathcal{C}=\set{\left<\sigma\right>}{\sigma\in\mathfrak{S}_p}$ denote the set of cyclic subgroups of $\mathfrak{S}_p$.
For $c\in\mathcal{C}$ we define $\Phi(c):=\# \beta(|c|)$ and $\mathcal{C}_c:= \set{\sigma\in\mathfrak{S}_p}{\left<\sigma\right>=c}$, the set of permutations, which generate the cyclic subgroup $c$. We have
\[
\Phi(c) = \#\mathcal{C}_c\qquad(c\in\mathcal{C}).
\]
For $c\in \mathcal{C}$, we denote
\[
\pi_{c} = \mathbb{P}(\Gamma = c|Z_1,\ldots,Z_n),
\]
which we want to approximate.
In our model we have (see \eqref{eq:defpp})
\begin{align}\label{eq:pic}
\pi_{c} \propto \frac{ I_{c}(\delta + n, D+U)}{I_{c}(\delta,D)}\qquad(c\in\mathcal{C}).
\end{align}
In order to find $\pi=(\pi_c;\, c\in\mathcal{C})$ let us consider $\tilde{\pi} = (\tilde{\pi}_\sigma;\, \sigma\in\mathfrak{S}_p)$, a probability distribution on $\mathfrak{S}_p$ such that
\begin{align}\label{eq:tpic}
\tilde{\pi}_{\sigma} \propto \frac{ I_{\left<\sigma\right>}(\delta + n, D+U)}{I_{\left<\sigma\right>}(\delta,D)}\qquad (\sigma\in\mathfrak{S}_p).
\end{align}
Since \eqref{eq:pic} and \eqref{eq:tpic} imply that $\tilde{\pi}_\sigma\propto \pi_{\left<\sigma\right>}$,
we have
\begin{align}\label{eq:pitpi}
\tilde{\pi}_\sigma &= \frac{ \pi_{\left<\sigma\right>}}{\sum_{c\in \mathcal{C}} \Phi(c) \pi_{c}}\qquad(\sigma\in\mathfrak{S}).
\end{align}
As before, let $(x_t)_{t\in\mathbb{N}}$ be a sequence of  i.i.d random transpositions distributed uniformly on $\mathcal{T}=\{(i,\,j)\in\mathfrak{S}_p\}$. We define a random walk on $\mathfrak{S}_p$ by 
\[
s_{t+1}=s_t\circ x_{t+1},\qquad (t=0,1,\ldots).
\]
Then, $(s_t)_t$ is an irreducible Markov chain with symmetric transition probability
\[
g(\sigma'|\sigma ) = 
\begin{cases} 
\frac{1}{\binom{p}{2}},  &\mbox{if }\, \sigma^{-1}\circ \sigma' \in\mathcal{T},\\
0, &\mbox{if }\, \sigma^{-1}\circ \sigma' \notin\mathcal{T}.
\end{cases}
\]
We note that $\left( \left<s_t\right>\right)_t$ is not a Markov chain on the space of cyclic subgroups. Indeed, it can be shown that the necessary conditions for $\left( f(s_t)\right)_t$ to be a Markov chain (see \cite[Eq. (3)]{MarkovFun}) are not satisfied for $f(\sigma):=\left<\sigma\right>$ if $p>4$. A remedy for this fact was introduced in \eqref{eq:defct}. Indeed, the sequence $(\left<s_t\right>)_t$ is very similar to the sequence $\left(c_t\right)_t$ defined previously. Both move along cyclic subgroups and their definitions are very similar. However, $(\left<s_t\right>)_t$ is not a Markov chain, whereas $(c_t)_t$ is a Markov chain. We took care of this problem by using the minimal generator $\nu(\cdot)$ as in  definition  \eqref{eq:defct} of $c_t$.

We use the Metropolis-Hastings algorithm with the above proposal distribution to approximate $\tilde{\pi}$.
\begin{alg}\label{alg2}
	Starting from a permutation $\sigma_0\in\mathfrak{S}_p$, repeat the following two steps for $t=1, 2,\ldots$:
	\begin{enumerate}
		\item Sample $x_t$ uniformly from the set $\mathcal{T}$ of all transpositions and set $\sigma^\prime =\sigma_{t-1}\circ x_t$;
		\item Accept the move $\sigma_t = \sigma^\prime$ with probability
		\[
		\min\left\{ 1, \frac{I_{\left<\sigma^\prime\right>}(\delta+n,D+U)\,\, I_{\left<\sigma_{t-1}\right>}(\delta,D) }{I_{\left<\sigma^\prime\right>}(\delta,D)\,\, I_{\left<\sigma_{t-1}\right>}(\delta+n,D+U)  } 
		\right\}.
		\]
		If the move is rejected, set $\sigma_t=\sigma_{t-1}$.
	\end{enumerate}
\end{alg}
\noindent By the ergodicity of the Markov chain $(\sigma_t)_t$ constructed above, as the number of steps $T\to\infty$, we have 
\begin{align}\label{eq:MHpp}
\frac{\sum_{t=1}^T {\bf 1}_{\sigma=\sigma_t }}{T}\stackrel{a.s.}{\longrightarrow} \tilde{\pi}_\sigma\qquad(\sigma\in\mathfrak{S}_p).
\end{align}
This fact allows us to develop a scheme for approximating the posterior probability $\pi$.
\begin{thm}\label{thm:Alg2}
	We have as $T\to\infty$,
	\begin{align}\label{eq:alg2}
	\frac{\frac{1}{\Phi(c) }\sum_{t=1}^T {\bf 1}_{c=\left<\sigma_t\right> } }{\sum_{t=1}^T  \frac{1}{\Phi(\left<\sigma_t\right>)}}\stackrel{a.s.}{\longrightarrow} \pi_c\qquad(c\in\mathcal{C}).
	\end{align}
\end{thm}
\begin{proof}
	Let us denote $n_\sigma^{(T)} = \sum_{t=1}^T {\bf 1}_{\sigma=\sigma_t }$, $\sigma\in\mathfrak{S}_p$. We have $T=\sum_{\sigma\in\mathfrak{S}_p} n_\sigma^{(T)}$ and $n_\sigma^{(T)}/T\stackrel{a.s.}{\longrightarrow}\tilde{\pi}_\sigma$.
	Moreover,
	\begin{align*}
	\frac{\frac{1}{\Phi(c) }\sum_{t=1}^T {\bf 1}_{c=\left<\sigma_t\right> } }{\sum_{t=1}^T  \frac{1}{\Phi(\left<\sigma_t\right>)}} &=  
	\frac{\frac{1}{\Phi(c) }\sum_{\sigma\in\mathcal{C}_c} n_\sigma^{(T)}}{\sum_{t=1}^T  \sum_{\gamma\in\mathcal{C}}
		\frac{1}{\Phi(\gamma)}{\bf 1}_{\gamma=\left<\sigma_t\right>} } \\
	& = 
	\frac{\frac{1}{\Phi(c) }\sum_{\sigma\in\mathcal{C}_c} \frac{n_\sigma^{(T)}}{T}}{\sum_{\gamma\in\mathcal{C}} \frac{1}{\Phi(\gamma)}  \sum_{\sigma\in\mathcal{C}_\gamma} \frac{n_\gamma^{(T)}}{T}} \\
	& \stackrel{a.s.}{\longrightarrow} 
	\frac{\frac{1}{\Phi(c) }\sum_{\sigma\in\mathcal{C}_c} \tilde{\pi}_\sigma }{\sum_{\gamma\in\mathcal{C}} \frac{1}{\Phi(\gamma)}  \sum_{\sigma\in\mathcal{C}_\gamma} \tilde{\pi}_\gamma}.
	\end{align*}
	Finally, by \eqref{eq:pitpi} we have 
	\[
	\frac{1}{\Phi(c) }\sum_{\sigma\in\mathcal{C}_c} \tilde{\pi}_\sigma  = \frac{ \pi_{c}}{\sum_{\gamma\in \mathcal{C}}\Phi(\gamma) \pi_{\gamma}}\propto \pi_c,
	\]	
	which completes the proof.
\end{proof}
In order to approximate the posterior probability $\pi$, we allowed the Markov chain to travel on the larger space $\mathfrak{S}_p$. In particular, each state $c\in\mathcal{C}$ was multiplied $\Phi(c)\geq 1$ times, where $\Phi(c)$ is the number of permutations generating $c$. This procedure should result in slower convergence to the stationary distribution in \eqref{eq:alg2}.
By comparing with \eqref{eq:MHpp}, we see that \eqref{eq:alg2} can be interpreted as follows: let us assign to each cyclic subgroup $c$ a weight $1/\Phi(c)\leq 1$. Then, the denominator $N_T:=\sum_{t=1}^T 1/\Phi(\left<\sigma_t\right>)$ can be thought of as an ``effective'' number of steps and the numerator is the number of  ``effective'' steps spent in state $c$. In general, for large $T$ we expect $N_T\ll T$ (see an example in Section \ref{sec:sim}).

\subsection{Model selection for $p=4$}
\label{selection:general}
Our numbering of colored models on four vertices is in accordance with \cite[Fig. 15 and 16, p. 674--675]{Ge11}. However, we identify models by the largest group with the same coloring $\Gamma^\ast$ rather than the smallest as in \cite{Ge11}. There are $30$ different subgroups of $\mathfrak{S}_4$, which generate $22$ different colored spaces.
Up to conjugacy (renumbering of vertices), there are $8$ different conjugacy classes. Within a conjugacy class, constants $(k_i,r_i,d_i)_{i=1}^L$ remain the same. Groups $\Gamma^\ast_k$ for $k=1,\ldots,17$ correspond to cyclic colorings.

\begin{table}
	\caption{Structure constants for all colorings with four vertices}
	\label{tab:tab1}
		\begin{tabular}{@{}lrrr@{}}
			\hline
			Group & $(k_i)$ & $(r_i)$ & $(d_i)$ \\ 
			\hline
			$\Gamma_1^\ast\,\,=\{\mathrm{id}\}$ & (1) & (4) & (1) \\ 
			\hline
			$\Gamma_2^\ast\,\,=\left\langle {(1,2)} \right\rangle$ &   (1,1) & (3,1) & (1,1)   \\ 
			$\Gamma_3^\ast\,\,=\left\langle {(1,3)} \right\rangle$ &  &  &  \\ 
			$\Gamma_4^\ast\,\,=\left\langle {(1,4)} \right\rangle$ & & & \\ 
			$\Gamma_5^\ast\,\,=\left\langle {(2,3)} \right\rangle$ &  &  & \\ 
			$\Gamma_6^\ast\,\,=\left\langle {(2,4)} \right\rangle$ &  &  &  \\ 
			$\Gamma_7^\ast\,\,=\left\langle {(3,4)} \right\rangle$ &  &  &  \\ 
			\hline
			$\Gamma_8^\ast\,\,=\left\langle {(1,2,3),\,\, (1,2)} \right\rangle$ &  (1,2) & (2,1) & (1,1)   \\ 
			$\Gamma_9^\ast\,\,=\left\langle {(1,2,4),\,\, (1,2)} \right\rangle$ &  & &  \\ 
			$\Gamma_{10}^\ast=\left\langle {(1,3,4),\,\, (1,3)} \right\rangle$ &  &  &  \\ 
			$\Gamma_{11}^\ast=\left\langle {(2,3,4),\,\, (2,3)} \right\rangle$ &  &  &  \\ 
			\hline
			$\Gamma_{12}^\ast=\left\langle {(1,2)(3,4)} \right\rangle$ & (1,1) & (2,2) & (1,1) \\ 
			$\Gamma_{13}^\ast=\left\langle {(1,3)(2,4)} \right\rangle$ &  &  &  \\ 
			$\Gamma_{14}^\ast=\left\langle {(1,4)(2,3)} \right\rangle$ &  &  &  \\ 
			\hline
			$\Gamma_{15}^\ast=\left\langle {(1,2,3,4),\,\,  (1,3)} \right\rangle$ & (1,1,2) & (1,1,1) & (1,1,1) \\ 
			$\Gamma_{16}^\ast=\left\langle {(1,2,4,3),\,\, (1,4)} \right\rangle$ &  &  &  \\ 
			$\Gamma_{17}^\ast=\left\langle {(1,3,2,4),\,\, (1,2)} \right\rangle$ &  &  &  \\ 
			\hline
			$\Gamma_{18}^\ast=\left\langle {(1,2),\,\, (3,4)} \right\rangle$ & (1,1,1) & (2,1,1) & (1,1,1) \\ 
			$\Gamma_{19}^\ast=\left\langle {(1,3),\,\, (2,4)} \right\rangle$ &  &  &  \\ 
			$\Gamma_{20}^\ast=\left\langle {(1,4),\,\, (2,3)} \right\rangle$ &  &  &  \\ 
			\hline
			$\Gamma_{21}^\ast=\left\langle {(1,2)(3,4),\,\,(1,4)(2,3)} \right\rangle$ & (1,1,1,1) & (1,1,1,1) & (1,1,1,1)  \\ 
			\hline
			$\Gamma_{22}^\ast=\mathfrak{S}_4$ & (1,3) & (1,1) & (1,1) \\
			\hline
	\end{tabular}
\end{table}

We apply our results and methods in order to do Bayesian model selection for the celebrated example of Frets' heads, \cite{Frets, Whitt90}. The head dimensions (length $L_i$ and breadth $B_i$, $i=1,2$) of 25 pairs of first and second sons were measured.
Thus we have $n=25$ and $p=4$. The following sample covariance matrix is obtained (we have $Z=(L_1,B_1,L_2,B_2)^\top$),
\[
U=\sum_{i=1}^n Z^{(i)}\cdot \transp{\mbox{$Z^{(i)}$}}= \begin{pmatrix}
2287.04 & 1268.84 & 1671.88 & 1106.68 \\
1268.84 & 1304.64 & 1231.48 & \,\,\,841.28 \\
1671.88 & 1231.48 & 2419.36 & 1356.96 \\
1106.68 & \,\,\,841.28 & 1356.96 & 1080.56
\end{pmatrix}
.
\]
We perform Bayesian model selection within all RCOP models, not just the ones corresponding to cyclic subgroups. In Table \ref{tab:tab1} we list all RCOP models on full graph with four vertices, along with corresponding structure constants. Structure constants remain the same within a conjugacy class, however the invariant measure $\varphi_{\Gamma}$ is always different. Since there are only $22$ such models, we calculate all exact posterior probabilities.
The Table \ref{tab:tab1} and the  invariant measures $\varphi_{\Gamma}$ were obtained by using Lemma \ref{rem:Useful}. 

In Table \ref{tab:tab3} we summarize the results when $\delta=3$ (a parameter of the prior distribution, Section \ref{sec:DY}), giving the three best coloring models with the highest posterior probability, for each given $D$. Results are very similar for $\delta=10$ and the given values of $D$. For comparison, the three best models according to BIC are $\Gamma_{19}^\ast$, $\Gamma_{13}^\ast$ and $\Gamma_8^\ast$ with the BIC $834.5$, $835.4$ and $835.5$ respectively. 

\begin{table}
	\caption{Posterior probabilities in Frets' heads for three best models, $\delta=3$ and given $D$.}
\label{tab:tab3} 
\begin{tabular}{@{}r|ll|ll|ll@{}}
	\hline
			$D$ & \multicolumn{2}{c}{Best model}	& \multicolumn{2}{c}{2nd best} & \multicolumn{2}{c}{3rd best} \\
			\hline
			$I_4$ &	$\Gamma_{22}^\ast$ &($95.2\%$) & $\Gamma_{16}^\ast$ &($2.5\%$) &	$\Gamma_{17}^\ast$ &($1.3\%$) \\
			$50I_4$ &	$\Gamma_{19}^\ast$ &($33.8\%$) &	$\Gamma_{13}^\ast$ &($29.6\%$)	& $\Gamma_{8}^\ast$ &($13.3\%$)\\
			$100I_4$ &	$\Gamma_{13}^\ast$ &($39.6\%$) &	$\Gamma_{19}^\ast$ &($29.8\%$) &	$\Gamma_{8}^\ast$ &($7.2\%$)	\\
			$1000I_4$ & $\Gamma_{1}^\ast$ &($38.9\%$) &	$\Gamma_{13}^\ast$ &($10.5\%$) & 	$\Gamma_{3}^\ast$ &($10.3\%$)	\\
			\hline
	\end{tabular}
\end{table}

For different values of $D=d I_4$, the only models that have highest posterior probability are the 4 models: $\Gamma_{22}^\ast = \mathfrak{S}_4$,
$\Gamma_{19}^\ast = \langle (1,3),\,\, (2,4) \rangle$,
$\Gamma_{13}^\ast = \langle {(1,3)(2,4)} \rangle$,
$\Gamma_1^\ast = \{\mathrm{id}\}$. 
These four subgroups form a path in the Hasse diagram
of subgroups of $\mathfrak{S}_4^\ast$, i.e.
$\Gamma_{22}^\ast \supset 
\Gamma_{19}^\ast \supset
\Gamma_{13}^\ast \supset	
\Gamma_1^\ast$.
Thus the four selected colorings, corresponding to the permutation groups 
are in some way consistent.
Moreover, each of them has a good statistical interpretation.
Let us interpret  models $\Gamma_{13}^\ast$ and $\Gamma_{19}^\ast$.
Recall the enumeration of vertices $(1,2,3,4) = (L_1, B_1, L_2, B_2)$.
The invariance with respect to the transposition ${(1,3)}$ means that $L_1$ is exchangeable with $L_2$ and, similarly, the invariance with respect to the transposition ${(2,4)}$ implies  exchangeability of $B_1$ and $B_2$. Both together correspond to the fact that sons should be exchangeable in some way.

We observe that only the $\Gamma_{22}^\ast$ model  appeared
in former attempts of model selection for Frets' heads data.
It was considered in \cite[Fig. S7 p.28 of the Supplementary
material]{Hel18b} with eleven other models.
Note that the only complete RCOP model selected in \cite{Ge11} (who used the Edwards-Havr\'anek model selection procedure) among the $9$ minimally accepted models on p. 676 of her article is $\Gamma_{10}^\ast$, which is not selected by our exact Bayesian procedure for any choice of $D=d \, I_4$. 


\section{Simulations}\label{practical}
Let the covariance matrix $\Sigma=(c_{ij})_{ij}\in\mathrm{Sym}^+(p;\real)$  be the symmetric circulant matrix defined by 
\[
c_{ij}=\begin{cases}
	1-|i-j|/p, & i\neq j, \\
	1+1/p, & i=j.
	\end{cases}
\]
It is easily seen that this matrix belongs to $\Pcone_{\langle\sigma^\ast\rangle}$ with $\sigma^\ast=(1,2,\ldots,p-1,p)$.

\subsection{First approach}

For $p=10$ and $n=20$, we sampled $Z^{(1)},\ldots, Z^{(n)}$ from the $\mathrm{N}_p(0,\Sigma)$ distribution and obtained $U=\sum_{i=1}^n Z^{(i)}\cdot \transp{\mbox{$Z^{(i)}$}}$ depicted in Fig. \ref{fig01} (b).

We run the Metropolis-Hastings algorithm starting from the group $\left<\sigma_0\right>=\{\mathrm{id}\}$ with hyper-parameters $\delta=3$ and $D=I_{10}$. After $1\,000\,000$ steps, the five most visited states are given in the Tab. \ref{tab:tab5}.

\begin{figure}
	\includegraphics[width=0.35\textwidth]{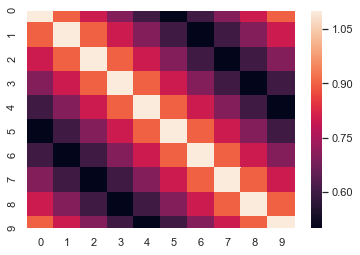}\,\,\,\,\,\,
	\includegraphics[width=0.35\textwidth]{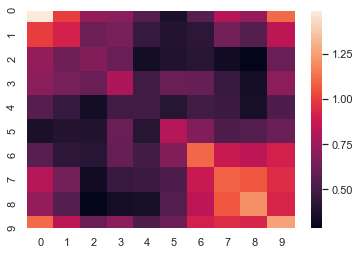}\\
	\makebox{{(a) \qquad\qquad\qquad\qquad\qquad\qquad\qquad\quad(b)}}
	\caption{Heat map of matrix $\Sigma$ (a) and matrix $U/n$ (b).}
	\label{fig01}
\end{figure}

\begin{table}
	\caption{Five most visited cyclic subgroups}
	\label{tab:tab5}
	\begin{tabular}{@{}rl@{}}
		\hline
		generator of a cyclic group & number of visits\\
		\hline
		(1, 2, 3, 4, 5, 6, 7, 8, 9,10) &  457\,725
		\\(1, 6, 2, 7)(3, 5, 9)(4, 8, 10)&  110\,677 
		\\(1, 6)(2, 7)(3, 5, 9)(4, 8, 10)&  51\,618
		\\(1, 7)(2, 6)(3, 5, 9)(4, 8, 10)&  40\,895
		\\(1, 2, 6, 7)(3, 5, 9)(4, 8, 10)&  34\,883\\
		\hline
	\end{tabular}
\end{table}

The Metropolis-Hastings (M-H) algorithm  recovered the true pattern of the covariance matrix. The acceptance rate was $2.5\%$ 
and the Markov chain visited $746$ different cyclic groups. The acceptance rate can be increased by a suitable choice of the hyper-parameters (e.g. for $D=10 I_{10}$ the acceptance rate is around $10\%$).

In order to grasp how randomness may influence results, we performed $100$ simulations, where each time we sample $Z^{(1)},\ldots, Z^{(n)}$ from $\mathrm{N}_p(0,\Sigma)$ and we run M-H for $100\,000$ steps with the same parameters as before. In Table \ref{tab:tabX} we present how many times a given cyclic subgroup was most visited during these $100$ simulations (second column). There were $53$ distinct cyclic subgroups, which were most visited at least in one of the $100$ simulations; below we present $10$ such subgroups. The average acceptance rate is $1.4\%$  (see the histogram in Fig. \ref{fig02}). 
\begin{table}
	\caption{Cyclic subgroups which were chosen by M-H algorithm most often}
	\label{tab:tabX}
\begin{tabular}{@{}lcl@{}}
	\hline
	\mbox{generator of a cyclic group} & \#\mbox{most visited} & \mbox{ARI}\\
	\hline
	(1, 2, 3, 4, 5, 6, 7, 8, 9,10) & 25 & 1.00
	\\(1, 3, 5, 7, 9)(2, 4, 6, 8,10)&  13 & 0.60
	\\(1, 2, 4, 3, 5, 6, 7, 9, 8, 10)&  3 & 0.43
	\\(1, 2, 4, 3, 5, 6, 7, 8, 9, 10) & 2 & 0.46
	\\(1, 3, 2, 4, 5, 6, 8, 7, 9, 10) & 2 & 0.43
	\\(1, 3, 5, 9, 2, 6, 8, 10, 4, 7) & 2 & 0.43
	\\(1, 4, 3, 5, 2, 6, 9, 8, 10, 7) & 2 & 0.35
	\\(1, 4, 5, 7, 8)(2, 3, 6, 9, 10) & 2 & 0.24
	\\(1, 8, 10, 9)(2, 7)(3, 5, 4, 6) & 2 & 0.19
	\\(1, 2, 10, 3)(4, 9)(5, 8, 6, 7) & 2 & 0.19\\
	\hline
\end{tabular}
\end{table}
When we regard colorings as partitions of the set $V\cup E$ according to group orbit decomposition, the two colorings may be compared using the so-called adjusted Rand index (ARI, see \cite{Partit}), a similarity measure comparing partitions which takes values between $-1$ and $1$, where $1$ stands for perfect match and independent random labelings have score close to $0$. In the third column of Table \ref{tab:tabX}, we give the adjusted Rand index between the colorings generated by given cyclic subgroup and the true coloring.

We see that groups which were most visited by the Markov chain have positive ARI and the true pattern was recovered in a quarter of cases. We stress that even though the colorings generated by $\left<(1, 2, 3, 4, 5, 6, 7, 8, 9,10)\right>$ and $\left< (1, 3, 5, 7, 9)(2, 4, 6, 8,10)\right>$ are very similar, the distance between these subgroups  is $9$, that is, the Markov chain $(C_t)_t$ needs at least $9$ steps to get from one subgroup to the other. We performed similar simulations for $n=p=10$ and the results were only slightly worse: the true pattern was recovered in $18$ out of $100$ runs of the algorithm.

This indicates that the Markov chain may encounter many local maxima and one should always tune the hyper parameters in order to have higher acceptance rate or to allow the Markov chain $(C_t)_t$ to make bigger steps.

\begin{figure}
	\centering
	\makebox{\includegraphics[width=0.4\textwidth]{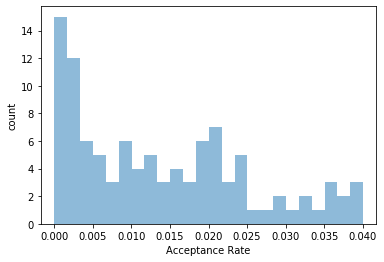}}
	\caption{\label{fig02}Histogram of acceptance rates in $100$ simulations of Metropolis-Hastings algorithm.}
\end{figure}

\subsection{Second approach}\label{sec:sim}
We performed $T=100\,000$ steps of Algorithm \ref{alg2} with $\sigma_0=\mathrm{id}$, $p=100$, $n=200$, $\delta=3$ and $D=I_{100}$. Let us note that for $p=100$, there are about $4\cdot10^{155}$  cyclic subgroups and this is the number of models we consider in our model search. 

We have used Theorem \ref{thm:Alg2} to approximate the posterior probability distribution $(\pi_c;\,c\in\mathcal{C} )$ (see \eqref{eq:defpp}).
The highest estimated posterior probability was obtained for $c^\ast:=\left<\sigma^\ast\right>$, where
\begin{align*}
\sigma^\ast = & 
(1, 2, 3, 4)(6, 8, 15)(7, 10, 9)(11, 16, 12)(13, 17, 14)(18, 19, 20, 22, 21)(23, 26) \\ 
& (24, 42, 28, 44)
(25, 31, 30, 32)(27, 34)(29, 37)(33, 45)(35, 39, 36, 40)\\
& (38, 47, 41, 48)(43, 51, 46, 49)(50, 52, 53, 54)(56, 58, 57)
(59, 66, 67)\\
&(60, 65, 63)(61, 62, 64)(68, 71, 72, 70, 69)(73, 93)(74, 77)(75, 98, 81, 100)\\
& (76, 84, 78, 83)(79, 85)(80, 94, 82, 91)(86, 92, 87, 90)(88, 96, 89, 97)(95, 99).
\end{align*}
The order of $c^\ast$ is $|c^\ast|=60$ and $\Phi(c^\ast)=16$. The estimate of the posterior probability $\pi_{c^\ast}$ is equal to (recall \eqref{eq:alg2})
\[
\frac{\frac{1}{\Phi(c^\ast) }\sum_{t=1}^T {\bf 1}_{c^\ast=\left<\sigma_t\right> } }{\sum_{t=1}^T  \frac{1}{\Phi(\left<\sigma_t\right>)}} \approx \frac{2\,361.5}{6\,381.5} \approx 37\%.
\]

The true covariance matrix $\Sigma$, the data matrix $U/n$ and the projection $\Pi_{c^\ast}(U/n)$ are illustrated in Fig. \ref{fig03}.

\begin{figure}
	\includegraphics[width=0.3\textwidth]{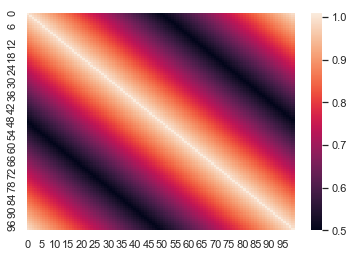}\ \ \,
	\includegraphics[width=0.3\textwidth]{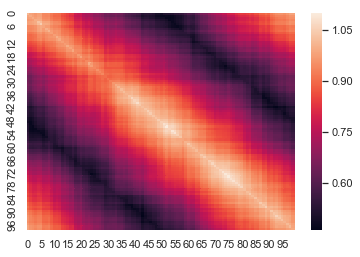}\ \ \,
	\includegraphics[width=0.3\textwidth]{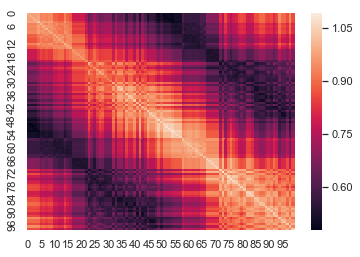}
	\makebox{{(a) \qquad\qquad\qquad\qquad\qquad\quad (b)   \quad\qquad\qquad\qquad\qquad\qquad(c)}}
	\caption{Heat map of matrix $\Sigma$ (a) and matrix $U/n$ (b) and projection of $U/n$ onto $\Zsp_{c^\ast}$.}
	\label{fig03}
\end{figure}
We visualize the performance of the algorithm on Fig \ref{fig04}. In red color, a sequence $\left( \sum_{t=1}^k  \frac{1}{\Phi(\left<\sigma_t\right>)}\right)_k$ is depicted, which can be thought of as  an ``effective'' number of steps of the algorithm (for an explanation, see the paragraph at the end of Subsection \ref{sec:sa}). In blue, we present a sequence $\left( \sum_{t=1}^k  \frac{1}{\Phi(\left<\sigma_t\right>)} {\bf 1}_{\left<\sigma_t\right>\neq\left<\sigma_{t-1}\right>}\right)_k$, which represents the number of weighted accepted steps, where the weight of the $k$th step equals $\frac{1}{\Phi(\left<\sigma_k\right>)}$. We restricted the plot to steps $k=1,\ldots, 10\,000$, because after $10\,000$ steps, the Markov chain
$(\sigma_t)_{10\,000\leq t\leq 100\,000}$ changed its state only $9$ times. For $k=100\,000$, the value of the blue curve is $25.75$, while the value of red one is $6\,381.5$.

\begin{figure}
	\includegraphics[width=0.5\textwidth]{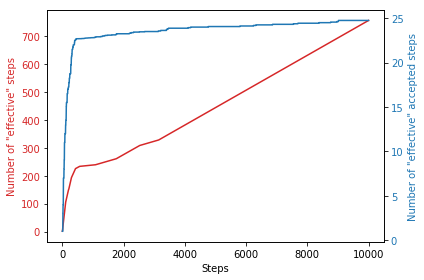}
	\caption{Number of ``effective'' steps (red) and number of ``effective'' accepted steps (blue).}
	\label{fig04} 
\end{figure}

The model suffers from poor acceptance rate, which could be improved by an appropriate choice of the hyper-parameter $D$ or by allowing the Markov chain to do bigger steps. 

%
\begin{acks}[Acknowledgments]
	The authors would like to thank Steffen Lauritzen for his interest and encouragements.
	We also thank M. Bogdan from Wroc{\l}aw University, A. Descatha from INSERM and Centre Hospitalier  Universitaire Angers and V. Seegers from Institut de Cancerologie de l'Ouest Nantes for explaining the specific nature of medical and genetic data.  The paper benefited from the comments of an anonymous referee to whom the authors are grateful. 
\end{acks}


%
%

\begin{supplement}
	
	\stitle{}
	
	\sdescription{}Supplement contains proofs and examples. 
We provide proofs of Theorems \ref{thm:intro}, \ref{thm:cyclic1}, \ref{thm:cyclic2} along with a background on representation theory that is needed to understand proofs. Moreover, we present proofs of Proposition \ref{pro:phi}, Theorems \ref{thm:IntegralPhi} and  \ref{thm:IntegralNoPhi}, an example to Section \ref{sec:25}, proof of Lemma \ref{lem:inv} and the real data example considered in \cite{Miller} and \cite{HL08}.
	
	In this document, references to equations are sometimes to the
	main file and sometimes to this supplementary file. For the reader's convenience we put a subindex $(\ )_{mf}$ to
	equation, section and theorem numbers referring to the main file.
	
	\section{Basics  of representation theory over reals}
	Representation theory has long been known to be very useful in statistics, cf. \cite{Diaconis}. However, the representation theory over $\real$ that we need in our work, is less known to the statisticians than the  standard one over $\complex$  (see Subsection \ref{sec:newHI} for a contrast
	between the theories over $\real$ and $\complex$).  In this section we recall some basic notions and results of the representation theory of groups over the reals. We intend to introduce the reader with all background needed to understand proofs of Theorem 1$_{mf}$ as well as Theorems 5$_{mf}$ and 6$_{mf}$. For  further details, the reader is referred to \cite{JPS77}.

	For a real vector space $V$, we denote by $\mathrm{GL}(V)$ the group of linear automorphisms on $V$.
	Let $G$ be a finite group. 
	\begin{defi}
		A function $\rho\colon G\to \mathrm{GL}(V)$ is called
		a representation of $G$ over $\real$ if it is a homomorphism, that is 
		\[
		\rho(g\, g')=\rho(g)\,\rho(g')\qquad (g,g'\in G).
		\]
		The vector space $V$ is called the representation space of $\rho$. 
	\end{defi}
	If $\dim V = n$, taking a basis $\{v_1, \dots, v_n\}$ of $V$, 
	we can identify $\mathrm{GL}(V)$ with the group $\mathrm{GL}(n; \real)$ of all $n \times n$ non-singular real matrices.
	Then a representation $\rho\colon G \to \mathrm{GL}(V)$ corresponds to a group homomorphism $B\colon G \to \mathrm{GL}(n; \real)$
	for which 
	\begin{equation}
		\label{action}
		\rho(g) v_j = \sum_{i=1}^n B_{ij}(g) v_i. 
	\end{equation}
	We call $B$ the matrix expression of $\rho$ with respect to the basis $\{v_1, \dots, v_n\}$.

	\begin{defi}
		A linear subspace $W\subset V$ is said to be $G$-invariant if 
		\[
		\rho(g)w\in W\qquad (w\in W,\, g\in G).
		\]
		A representation $\rho$ is said to be irreducible if the only $G$-invariant subspaces are non-proper, that is, whole $V$ and $\{0\}$.
		A restriction of $\rho$ to a $G$-invariant subspace $W$ is a subrepresentation.
		Two representations, 
		$\rho\colon G \to \mathrm{GL}(V)$ and $\rho'\colon G \to \mathrm{GL}(V')$ are equivalent
		if there exists an isomorphism of vector spaces $\ell\colon V\mapsto V'$ with
		\[
		\ell(\rho(g)v) = \rho'(g)\ell(v)\qquad (v\in V,\, g\in G).
		\]
	\end{defi}
	We note that a group homomorphism $B\colon G \to \mathrm{GL}(n; \real)$ 
	defines a representation of $G$ on $\real^n$
	naturally.
	We see that
	$B$ is a matrix expression of a representation $(\rho, V)$ 
	if and only if $B$ and $\rho$ are equivalent via the map 
	$\ell\colon \real^n \owns (x_i)_{i=1}^n  \mapsto \sum_{i=1}^n x_i v_i \in V$, that is, $\ell(B(g)\underline{x}) = \rho(g)\ell(\underline{x})$ for $\underline{x}\in\real^n$. Here $\{v_1,\ldots,v_n\}$ denotes a fixed basis of $V$.
	Therefore,  
	two representations $(\rho, V)$ and $(\rho', V')$ are equivalent 
	if and only if 
	they have the same matrix expressions with respect to appropriately chosen bases.
	We shall write $\rho \sim B$ if $\rho$ has a matrix expression $B$ with respect to some basis.
	
	Let $(\rho, V)$ be a representation of $G$, and $B\colon G \to \mathrm{GL}(n; \real)$ be a matrix expression of $\rho$
	with respect to a basis $\{v_1, \dots, v_n\}$ of $V$.
	Then it is known that  the function 
	$\chi_\rho\colon G \owns g \mapsto \mathrm{Tr}\,B(g) = \sum_{i=1}^n B_{ii}(g) \in \real$
	is independent of the choice of the basis $\{v_1, \dots, v_n\}$. 
	The function $\chi_\rho$ is called a character of the representation $\rho$. 
	The function $\chi_\rho$ characterizes the representation $\rho$ in the following sense.
	\begin{lem} \label{lem:character}
		Two representations $(\rho, V)$ and $(\rho', V')$ of a group $G$ are equivalent if and only if $\chi_\rho = \chi_{\rho'}$. 
	\end{lem} 
	We apply this lemma  in practice to know whether two given representations are equivalent or not.
	
	It is known that, for a finite group $G$, 
	the set $\Lambda(G)$ of equivalence classes of irreducible representations of $G$ is a finite set.
	We fix the group homomorphisms 
	$B_{\alpha}\colon G \to \mathrm{GL}(k_{\alpha}; \real)$, $\alpha \in A,$
	indexed by a finite set $A$
	so that
	$\Lambda(G) = \set{[B_\alpha]}{\alpha \in A}$, 
	where $[B_\alpha]$ denotes the equivalence class of $B_\alpha$.
	
	Let $(\rho, V)$ be a representation of $G$.
	Then
	there exists a $G$-invariant inner product on $V$.
	In fact, from any inner product $\cpling{\cdot}{\cdot}_0$ on $V$,
	one can define such an invariant inner product $\cpling{\cdot}{\cdot}$ by
	$\cpling{v}{v'} := \sum_{g \in G} \cpling{ \rho(g)v}{\rho(g)v'}_0$ for  $v, v' \in V$.  
	In what follows, we fix a $G$-invariant inner product on $V$. 
	
	If $W$ is a $G$-invariant subspace, 
	the orthogonal complement $W^\bot$ 
	is also a $G$-invariant subspace. 
	Thus, any representation $\rho$ can be decomposed into a finite number of irreducible subrepresentations
	\begin{align}\label{eq:IrrepDecomposition}
		\rho=\rho_1\oplus \ldots\oplus \rho_K
	\end{align}
	along the orthogonal decomposition $V = V_1 \oplus \dots \oplus V_K$,
	where $\rho_i$ is the restriction of $\rho$ to the $G$-invariant subspace $V_i$, $i=1,\ldots,K$. 
	Let $r_\alpha$ be the number of subrepresentations $\rho_i$  such that $\rho_i\sim B_\alpha$. 
	Although the irreducible decomposition \eqref{eq:IrrepDecomposition} of $V$ is not unique in general, 
	$r_\alpha$ is uniquely determined.
	We have 
	\begin{align} \label{eqn:Balpha_decomp}
		\rho &\sim \bigoplus_{r_\alpha >0} B_\alpha^{\oplus r_\alpha}, 
	\end{align}
	where $\sum_{r_{\alpha>0}}r_{\alpha}=K.$  
	To see this,
	let $V(B_\alpha)$ be the direct sum of subspaces $V_i$ for which $\rho_i \sim B_\alpha$.
	The space $V(B_\alpha)$ is called the $B_\alpha$-component of $V$.
	If $r_\alpha>0$,
	gathering an appropriate  basis of each $V_i$, 
	the matrix expression of 
	the subrepresentation of $\rho$ on $V(B_\alpha)$ becomes (recall that $B_\alpha(g)\in 
	\mathrm{GL} (k_\alpha;\real)$)
	\[
	B_\alpha(g)^{\oplus r_\alpha} =
	\begin{pmatrix} B_\alpha(g) & & & \\ & B_\alpha(g) & & \\ & & \ddots & \\ & & & B_\alpha(g) \end{pmatrix}
	= I_{r_\alpha} \otimes B_\alpha(g) 
	\in \mathrm{GL}(r_\alpha k_\alpha; \real)
	\qquad (g \in G).
	\]
	Moreover,
	$V$ is decomposed as $V = \bigoplus_{r_\alpha >0} V(B_\alpha)$.
	Therefore, 
	taking a basis of $V$ by gathering the bases of $V(B_\alpha)$, 
	we obtain \eqref{eqn:Balpha_decomp}.
	
	\section{The proof of Theorem 1$_{mf}$}
	In this section we apply general results on representation theory from previous section to the mapping $\sigma \mapsto R(\sigma)$ defined in $(3)_{mf}$. 
	
	Let $\Gamma$ be a subgroup of the symmetric group $\mathfrak{S}_p$. By definition, we have $R\colon \Gamma\to \mathrm{GL}(p;\real)$ and $R(\sigma\circ\sigma')=R(\sigma)\cdot R(\sigma')$ for all $\sigma, \sigma'\in\Gamma$. Thus, $R$ is a representation of $\Gamma$ over $\real$.

	We will show, in this section, that 
	for $R$, as for all representations of a finite group, 
	through an appropriate change of basis, matrices $R(\sigma)$, $\sigma \in \Gamma$, can be simultaneously  written as  block diagonal matrices with the number and dimensions of these block matrices being the same for all $\sigma\in \Gamma$. This, in turn, will imply that any matrix in $\Zsp_{\Gamma}$ can be written under the form described by Theorem 1$_{mf}$. For reader’s convenience we repeat its  statement below.
	\begin{thm}\label{thm:intro}
		Fix a permutation subgroup $\Gamma\subset\mathfrak{S}_p$. Then, there exist constants $L\in\mathbb{N}$, $(k_i, d_i, r_i)_{i=1}^L$ and orthogonal matrix $U_\Gamma$ such that if $X\in\mathcal{Z}_\Gamma$, i.e. $X\in\mathrm{Sym}(p;\real)$ and
		\[
		X_{ij}=X_{\sigma(i)\sigma(j)}\qquad(\sigma\in\Gamma,\, i,j\in\{1,\ldots,p\}),
		\]
		then
		\begin{equation}
			\label{genform}
			X
			= U_\Gamma\cdot
			\begin{pmatrix} 
				M_{\mathbb{K}_1}(x_1) \otimes I_{k_1/d_1} & & & \\
				& M_{\mathbb{K}_2}(x_2) \otimes I_{k_2/d_2} & & \\
				& & \ddots & \\
				& & & M_{\mathbb{K}_L}(x_L) \otimes I_{k_L/d_L} 
			\end{pmatrix}\cdot \transp{U_\Gamma},
		\end{equation}
		where $M_{\mathbb{K}_i}(x_i)$ is a real matrix representation of 
		an $r_i \times r_i$ Hermitian matrix $x_i$ with entries in  $\mathbb{K}_i=\real, \complex$ or $\quat$, $i=1,\ldots, L$,
		and $A \otimes B$ denotes the Kronecker product of matrices $A$ and $B$. 
	\end{thm}

	\subsection{Irreducible decomposition of representation $R$}
	
	Regarding $\rho(\sigma)=R(\sigma) \in \mathrm{GL}(\real^p)$ as an operator on $V = \real^p$ via the standard basis $v_i=e_i \in \real^p$, $i=1,\ldots, p$, we see that  \eqref{action} holds trivially with $B=R$.
	
	We will apply \eqref{eqn:Balpha_decomp} to $G=\Gamma \subset \mathfrak{S}_p$ and $(\rho, V)=(R, \real^p)$. If we let $\set{\alpha \in A}{r_\alpha >0} =: \{\alpha_1,\alpha_2, \dots, \alpha_L\}$ 
	and if we denote by $U_{\Gamma}$ 
	an orthogonal matrix whose column vectors form  orthonormal bases
	of $V(B_{\alpha_1}),\ldots,V(B_{\alpha_L})$ successively, 
	then for $\sigma\in \Gamma$, we have
	\begin{equation} \label{eqn:RU_diagonal}
		\transp{U_\Gamma} \cdot  R(\sigma) \cdot U_\Gamma
		= \begin{pmatrix}  I_{r_1} \otimes B_{\alpha_1}(\sigma) & & & \\ & I_{r_2} \otimes B_{\alpha_2}(\sigma) & & \\ & & \ddots & \\
			& & & I_{r_L} \otimes B_{\alpha_L}(\sigma) \end{pmatrix}.
	\end{equation}
	Note that, since the left hand side of \eqref{eqn:RU_diagonal} is an orthogonal matrix, matrices $B_{\alpha}(\sigma)$, $\alpha\in A$, are orthogonal.
	In the general case, $B_{\alpha}(g)$ are orthogonal if we work with a $G$-invariant inner product. Note that the usual inner product on $V = \real^p$  
	is clearly $\Gamma$-invariant. 
	
	Example below gives an illustration of the representation $R$ and also an illustration of all the notions and results we already stated.
	\begin{exa} \label{ex3}
		Let $p=4$
		and let $\Gamma=\{ \mathrm{id}, (1,2)(3,4)\}$ 
		be the subgroup of $\mathfrak{S}_4$ generated by $\sigma= (1,2)(3,4)$.
		The matrix representation of $\sigma$ in the standard basis $(e_i)_i$ of $\real^4$ is
		\[
		R(\sigma)=
		\begin{pmatrix} 0& 1& 0& 0 \\ 1 &0 &0 &0 \\ 0& 0& 0 &1\\ 0 &0& 1 &0
		\end{pmatrix},
		\]
		which has  the two eigenvalues $1$ and $-1$ with multiplicity 2 for each.
		We choose  the following orthonormal eigenvectors of $R(\sigma)$:
		\[
		u_1=\frac1{\sqrt{2}}(e_1+e_2),
		u_2=\frac1{\sqrt{2}}(e_3+e_4),
		u_3= \frac1{\sqrt{2}}(e_1-e_2), 
		u_4= \frac1{\sqrt{2}}(e_3-e_4)
		\]
		and let $U_\Gamma=(u_1,u_2,u_3,u_4)$.
		The corresponding eigenspaces $V_i=\real u_i$ 
		are invariant under $R(\sigma)$ and  $R(\mathrm{id})=I_4$.
		As $V_i$, $i=1,\ldots,4$, are 1-dimensional, the subrepresentations
		defined by 
		\[
		\rho_i(\gamma)=R(\gamma)|_{V_i}\qquad (\gamma\in \Gamma)
		\]
		are irreducible.
		We have the decomposition \eqref{eq:IrrepDecomposition} of $R$:
		\[
		R =\rho_1\oplus \rho_2 \oplus \rho_3\oplus \rho_4.
		\]
		The  matrix expressions of $\rho_1$ and
		$\rho_2$
		are equal to $B_1(\gamma)=(1)$
		for all $\gamma\in \Gamma$, since $\rho_i(\gamma)v=v$ for $v\in V_i$,
		$i=1,2$.
		We have $r_1=2$. 
		
		The matrix expressions of  $\rho_3$
		and  $\rho_4$ are both  
		equal to $B_2(\gamma)=\mathrm{sign}(\gamma)$
		for all $\gamma\in \Gamma$, since $\rho_i(\mathrm{id})v=v$ 
		and  $\rho_i(\sigma)v=-v$ for $v\in V_i$ for $i=3,4$.
		We have $r_2=2$. 
		
		The representations $\rho_1$ and $\rho_3$ are not equivalent, which can be seen by looking at the characters:  $\chi_{\rho_1}=1$,
		$\chi_{\rho_3}(\gamma)=\mathrm{sign}(\gamma)$, which are not equal.

		In the basis $u_1, u_2, u_3,u_4$, the matrix of $R(\gamma)$ is (compare with \eqref{eqn:RU_diagonal})
		\[
		\begin{pmatrix}
			1 &0&0&0\\0 &1&0&0\\
			0&0&  \mathrm{sign}(\gamma)&0\\
			0&0&0&  \mathrm{sign}(\gamma)
		\end{pmatrix} 
		= B_1(\gamma)^{\oplus 2}  \oplus B_2(\gamma)^{\oplus 2} =
		\transp{U_\Gamma}\cdot R(\gamma)\cdot U_\Gamma.
		\]
		This is the decomposition \eqref{eqn:Balpha_decomp} of $R$
		in the basis $(u_1, u_2, u_3,u_4)$.
	\end{exa}

	\subsection{Block diagonal decomposition of $\Zsp_{\Gamma}$}
	So far, we have shown that through an appropriate change of basis,
	the representation  
	$(R,\real^p)$ of $\Gamma$
	can be expressed as the direct sum  \eqref{eqn:Balpha_decomp} of irreducible subrepresentations. We now want to turn our attention to the elements of $\Zsp_{\Gamma}.$
	
	A linear operator $T\colon V \to V$ is said to be an intertwining operator of the representation $(\rho, V)$
	if $T \circ \rho(g) = \rho(g) \circ T$ holds for all $g \in G$.
	In our context,
	since $(4)_{mf}$ can be rewritten as
	\begin{align}\label{eq:newZ}
		\Zsp_{\Gamma}
		= \set{x \in \mathrm{Sym}(p;\real)}
		{x \cdot R(\sigma) = R(\sigma)\cdot  x \mbox{ for all } \sigma \in \Gamma},
	\end{align}
	$\Zsp_\Gamma$ is the set of symmetric intertwining operators of the representation 
	$(R, \real^p)$. 
	
	Let $\mathrm{End}_{\Gamma}(\real^p)$ denote the set of all intertwining operators of the representation $(R, \real^p)$ of $\Gamma$. Recall that the set $A$ enumerates the elements of $\Lambda(\Gamma)$, the finite set of all equivalence classes of irreducible representations of $\Gamma$.
	From \eqref{eqn:Balpha_decomp} and \eqref{eqn:RU_diagonal}, it is clear that
	to study $\mathrm{End}_{\Gamma}(\real^p)$, it is sufficient to study  the sets,  
	\[
	\mathrm{End}_\Gamma(V_\alpha) = 
	\set{T \in \mathrm{Mat}(k_\alpha, k_\alpha;\real)}{T \cdot B_\alpha(\sigma) = B_{\alpha}(\sigma) \cdot T \mbox{ for all }\sigma \in \Gamma},
	\]
	$\alpha \in A$,
	of all intertwining operators of the irreducible representation $B_{\alpha}$,
	where $V_\alpha := \real^{k_\alpha}$ is the representation space of $B_\alpha$
	equipped with a $\Gamma$-invariant inner product. 
	Indeed, we have $V(B_\alpha)=I_{r_\alpha}\otimes V_\alpha$.
	

	The actual formula for $B_{\alpha}(\sigma)$ obviously depends on the choice of $U_\Gamma$ and hence, on the choice of orthonormal basis of $\real^p$. To ensure simplicity of formulation of our next result (Lemma \ref{lem:BalphaProperty}), we will work with special orthonormal bases
	of $V(B_{\alpha_1}),\ldots,V(B_{\alpha_L})$, which together constitute a basis of $\real^p$. Such bases always exist and will be defined in the next section.
	Usage of these bases is not indispensable for the proof of Theorem \ref{thm:intro}, but simplifies it greatly.
	
	The result from \cite[Page 108]{JPS77} implies that, since the representation $B_\alpha$ is irreducible,
	the space $\mathrm{End}_{\Gamma}(V_\alpha)$ is 
	isomorphic either to $\real$, $\complex$, or the quaternion algebra $\quat$.
	Let
	\[
	f_{\alpha}\colon \mathbb{K}_\alpha \to \mathrm{End}_{\Gamma}(V_\alpha),
	\]
	denote this isomorphism, 
	where $\mathbb{K}_\alpha$ is $\real$, $\complex$, or $\quat$. Let 
	\[
	d_\alpha := \dim_{\real} \mathrm{End}_{\Gamma }(V_\alpha) =  \dim_{\real}\, \mathbb{K}_\alpha\in\{1,2,4\}.
	\]
	The representation space $V_\alpha$
	becomes a vector space over $\mathbb{K}_\alpha$ of dimension $k_{\alpha} /d_{\alpha}$
	via 
	\[
	q \cdot v := f_{\alpha}(q) v \qquad (q \in \mathbb{K}_\alpha,\,v \in V_\alpha).
	\]
	Clearly the space $\real I_{k_\alpha}$ of scalar matrices is contained 
	in $\mathrm{End}_\Gamma(V_\alpha)$. 
	If  $d_\alpha = 1 = \dim_{\real} \real I_{k_\alpha}$, 
	we have $\mathrm{End}_\Gamma(V_\alpha) = \real I_{k_\alpha}$. 
	Further, if $d_\alpha=2$, 
	take a $\complex$-basis $\{v_1, \dots, v_{k_\alpha/2}\}$ of $V_\alpha$
	in such a way that
	$\{v_1, \dots, v_{k_\alpha/2},  i  \cdot v_1, \dots, i \cdot v_{k_\alpha/2}\}$ is 
	an orthonormal $\real$-basis of $V_\alpha$. We identify $\real^{k_\alpha}$ and $V_\alpha$ via this $\real$-basis. Then, 
	the action of $q=a + b i \in \complex$ on
	$w 
	\in \real^{k_\alpha}\simeq V_\alpha$
	is expressed as
	\[
	q\cdot w= \begin{pmatrix} a I_{k_\alpha/2} & -b I_{k_\alpha/2} \\ b I_{k_\alpha/2} & a I_{k_\alpha/2} \end{pmatrix} w
	= \left\{M_{\complex}(a+ bi) \otimes I_{k_\alpha/2}\right\}w.
	\]
	Thus, if $d_\alpha=2$, then
	\[
	\mathrm{End}_{\Gamma}(V_\alpha) = \set{M_{\complex}(q) \otimes I_{k_\alpha/2}}{q\in\complex}=	M_{\complex}(\complex) \otimes I_{k_\alpha / 2}.
	\]
	Similarly,
	when $\mathbb{K}_\alpha = \quat$, 
	take an $\quat$-basis $\{v_1, \dots, v_{k_\alpha/4}\}$ of $V_\alpha$
	so that 
	\[
	\{v_1, \dots, v_{k_\alpha/4}, i \cdot v_1, \dots, i \cdot v_{k_\alpha/4},  j \cdot v_1, \dots, j \cdot v_{k_\alpha/4}, 
	k \cdot v_1, \dots, k \cdot v_{k_\alpha/4}\}
	\]
	is an orthonormal $\real$-basis of $V_\alpha$.
	The action of $Q \in \quat$ on $V_\alpha$ is expressed as $M_\quat(Q) \otimes I_{k_\alpha/4}$
	with respect to this basis.
	
	In this way we have proved the following result.
	
	\begin{lem} \label{lem:BalphaProperty}
		For each $\alpha \in A$,  one has
		\begin{align} \label{eqn:Kalpha}
			\mathrm{End}_\Gamma(V_\alpha) = 	M_{\mathbb{K}_{\alpha}}(\mathbb{K}_{\alpha}) \otimes I_{k_\alpha / d_\alpha}.
		\end{align}
	\end{lem}

	For the proof of Theorem \ref{thm:intro}, we will need the following result.
	\begin{lem} \label{lem:CommutRel}
		Let $i,j = 1, 2, \dots, L$,
		and 
		assume that $Y \in \mathrm{Mat}(r_i k_i,\, r_j k_j; \real)$ satisfies the condition
		\begin{equation} \label{eqn:CommutRel}
			[I_{r_i} \otimes B_i(\sigma)] \cdot Y = Y \cdot [I_{r_j} \otimes B_j(\sigma)] \qquad (\sigma \in \Gamma).
		\end{equation}
		If $i = j$, 
		then there exists $C \in \mathrm{Mat}(r_i,\, r_i; \mathbb{K}_i)$
		such that $Y = M_{\mathbb{K}_i}(C) \otimes I_{k_i/d_i}$.
		On the other hand, if $i \ne j$, then $Y = 0$.
	\end{lem}
	\begin{proof}
		Let us consider a block decomposition of $Y$ as
		$$
		Y = \begin{pmatrix} Y_{11} & Y_{12} & \dots & Y_{1,r_j} \\
			Y_{21} & Y_{22} & \dots & Y_{2, r_j} \\
			\vdots & \vdots & \ddots & \vdots \\
			Y_{r_i, 1} & Y_{r_i, 2} & \dots & Y_{r_i, r_j} \end{pmatrix},
		$$
		where each $Y_{ab}$ is a $k_i \times k_j$ matrix.
		Then \eqref{eqn:CommutRel} implies that 
		\begin{equation} \label{eqn:CommutRel2}
			B_i(\sigma) \cdot Y_{ab} = Y_{ab} \cdot B_j(\sigma) \qquad (\sigma \in \Gamma) 
		\end{equation}
		for all $a,b$.
		If $i = j$, then $Y_{ab} \in \mathrm{End}_\Gamma(\real^{k_i})$, 
		so that there exists $C_{ab} \in \mathbb{K}_i$ for which
		$Y_{ab} = M_{\mathbb{K}_i}(C_{ab}) \otimes I_{k_i/d_i}$ 
		thanks to  Lemma \ref{lem:BalphaProperty}.
		Let us consider the case $i \ne j$.
		Eq. \eqref{eqn:CommutRel2} tells us that $\mathrm{Ker}\, Y_{ab} \subset \real^{k_j}$ is a $\Gamma$-invariant subspace,
		which then equals $\{0\}$ or $\real^{k_j}$ because of the irreducibility of $B_j$.
		Similarly, since $\mathrm{Image}\,Y_{ab} \subset \real^{k_i}$ is a $\Gamma$-invariant subspace by \eqref{eqn:CommutRel2},
		$\mathrm{Image}\,Y_{ab}$ equals $\{0\}$ or $\real^{k_i}$.  
		Now suppose that $Y_{ab} \ne 0$. 
		Then $\mathrm{Ker}\,Y_{ab} = \{0\}$ and $\mathrm{Image}\,Y_{ab} = \real^{k_i}$ by the argument above,
		and it means that $Y_{ab}$ induces an isomorphism from $(B_j, \real^{k_j})$ onto
		$(B_{i }, \real^{k_i})$. 
		But this contradicts the fact that 
		the representations $B_i$ and $B_j$ are not equivalent for $i \ne j$. 
		Hence we get $Y_{ab} = 0$.
	\end{proof}
	
	\begin{proof}[Proof of Theorem \ref{thm:intro}]
		Take $y \in \transp{U_{\Gamma}} \cdot \Zsp_\Gamma \cdot U_\Gamma$ and consider the block decomposition 
		of $y$ as
		\[
		y = \begin{pmatrix} Y_{11} & Y_{12} & \dots & Y_{1L} \\ Y_{21} & Y_{22} & \dots & Y_{2 L} \\
			\vdots & \vdots & \ddots & \vdots \\ Y_{L1} & Y_{L2} & \cdots & Y_{LL} \end{pmatrix}
		\]
		with $Y_{ij} \in \mathrm{Mat}(r_i k_i, \,r_j k_j;\real)$.
		Then $x := U_\Gamma \cdot y \cdot \transp{U_\Gamma}$ belongs to $\Zsp_\Gamma$,
		so that \eqref{eq:newZ} implies
		\[
		R(\sigma) \cdot U_\Gamma \cdot y \cdot \transp{U_\Gamma} \cdot \transp{R(\sigma)}
		= U_\Gamma \cdot y \cdot \transp{U_\Gamma}
		\] 
		for $\sigma \in \Gamma$, and this equality is rewritten as
		$$
		[\transp{U_\Gamma} \cdot R(\sigma) \cdot U_\Gamma] \cdot y
		= y \cdot [\transp{U_\Gamma} \cdot R(\sigma) \cdot U_\Gamma]. 
		$$
		By (\ref{eqn:RU_diagonal}), we have
		\[
		[I_{r_i} \otimes B_i(\sigma)] \cdot Y_{ij} = Y_{ij} \cdot [I_{r_j} \otimes B_j(\sigma)].
		\]
		Lemma \ref{lem:CommutRel} tells us that $Y_{ij} = 0$ if $i \ne j$, 
		and that $Y_{ii} = M_{\mathbb{K}_i}(x_i) \otimes I_{k_i/d_i}$ with some $x_i \in \mathrm{Mat}(r_i,\,r_i; \mathbb{K}_i)$.
		Since $y$ is a symmetric matrix, the block $Y_{ii}$ is also symmetric, which implies that 
		$x_i \in \mathrm{Herm}(r_i; \mathbb{K}_i)$. Actually, the map
		$\iota\colon \bigoplus_{i=1}^L  \mathrm{Herm}(r_i; \mathbb{K}_i) \owns (x_i)_{i=1}^L
		\mapsto X \in \Zsp_\Gamma$
		given by \eqref{genform}
		gives a Jordan algebra isomorphism.
	\end{proof}

	\subsection{A comparison to the representation theory over the complex number field}\label{sec:newHI}
	Theorem \ref{thm:intro} has a much simpler counterpart
	in the representation theory over $\complex$,
	which we state in a spirit of  \cite{Sh13} and \cite{Sh132}.
	Let $\Gamma$ be a subgroup of $\mathfrak{S}_p$. 
	We regard the natural representation $R$ of $\Gamma$ as a complex representation
	$R\colon \Gamma \to \mathrm{GL}(p; \complex)$.
	Assume that $R$ is decomposed 
	as $R \sim \bigoplus_{k=1}^K \vartheta_k^{\oplus s_k}$, 
	where $\vartheta_k\colon \Gamma \to \mathrm{GL}(m_k; \complex)$, $k=1, \ldots, K$,
	are mutually inequivalent irreducible complex representations.
	Let $W_\Gamma^{\complex}$ be the vector space consisting of 
	$A \in \mathrm{Mat}(p,p; \complex)$
	such that $A \cdot R(\sigma) = R(\sigma) \cdot A$ for all $\sigma \in \Gamma$.
	Then there exists a unitary $p\times p$ matrix $T_\Gamma$ for which 
	all the matrices $A \in W_{\Gamma}^{\complex}$ are simultaneously diagonalized as
	\begin{equation} \label{eqn:W_Gamma}
		T_\Gamma^* \cdot A \cdot T_\Gamma = 
		\begin{pmatrix} a_1 \otimes I_{m_1} & & & \\ & a_2 \otimes I_{m_2} & & \\
			& & \ddots & \\ & & & a_K \otimes I_{m_K} \end{pmatrix}, 
		\quad 
		\begin{array}{l}	a_k \in \mathrm{Mat}(s_k, \complex),\\
			k=1, \dots, K.
		\end{array}
	\end{equation}
	Precisely, the diagonal blocks in the right-hand side are of the form 
	$I_{m_k} \otimes a_k$ in \cite{Sh132}, but the difference can be made up by an appropriate permutation of the columns of $T_\Gamma$.
	Clearly the constants $(m_k, s_k)_k$ correspond to our structure constants $(k_i, r_i)_i$,
	while we can consider that 
	a complex counterpart for $d_i$ takes always the value $1$.
	Since $T_\Gamma$ is a unitary matrix, 
	we see that if $A$ is Hermitian, then 
	the corresponding matrices $a_k$, $k=1, \ldots, K$, are also Hermitian.
	This fact together with (\ref{eqn:W_Gamma}) is efficiently utilized
	in a study of complex covariance matrices with group symmetry in \cite{STW16}.
	On the other hand, 
	even though $A \in W_\Gamma^{\complex}$ is a real matrix, 
	the matrices $a_k$ are not necessarily real, 
	as \cite{Sh13} seem to misunderstand implicitly. 
	For instance,
	let us consider the case where $p=3$ 
	and $\Gamma \subset \mathfrak{S}_3$ is a cyclic group generated by 
	$\begin{pmatrix} 1 & 2 & 3 \end{pmatrix}$.
	Then $A \in W_\Gamma^{\complex}$ is of the form 
	$\begin{pmatrix} a & b & c \\ c & a & b \\ b & c & a \end{pmatrix}$
	with $a,b,c \in \complex$.
	Taking $T_\Gamma := \frac{1}{\sqrt{3}}
	\begin{pmatrix} 1 & 1 & 1 \\ 1 & \omega & \lbar{\omega} \\ 
		1 & \lbar{\omega} & \omega \end{pmatrix}$
	with $\omega := e^{2\pi \iu /3}$,
	we have
	$$
	T_\Gamma^\ast
	\begin{pmatrix} a & b & c \\ c & a & b \\ b & c & a \end{pmatrix}
	T_\Gamma
	= \begin{pmatrix} a_1 & & \\ & a_2 & \\ & & a_3 \end{pmatrix},
	$$
	where $a_1 := a+ b  + c$, $a_2 := a + b \omega + c \lbar{\omega}$, 
	and $a_3 := a + b \lbar{\omega} + c \omega$.
	In this case $m_k = s_k = 1$ for $k=1,2,3$
	(confer \cite[Remark 3.1]{Sh13}).
	Even if $a, b, c$ are real, the right-hand side above is not necessarily real but
	of the form $\mathrm{diag}(a_1, a_2, \lbar{a_2})$ with $a_1 \in \real$
	and $a_2 \in \complex$.
	Furthermore, if the matrix $A$ is symmetric, that is, real and Hermitian,
	then right-hand side becomes $\mathrm{diag}(a_1, a_2, a_2)$ with $a_1, a_2 \in \real$
	as is seen from Theorem \ref{thm:intro} with 
	$(k_1, k_2) = (1, 2),\, (r_1, r_2) = (1,1)$ and $(d_1, d_2) = (1,2)$.
	This observation tells us that
	the constants $(m_k, s_k)_k$ from complex representation theory
	are not sufficient for the description of the simultaneous diagonalization of 
	real symmetric matrices with group symmetry.

	\section{Additions to Section 2$_{mf}$}
	\subsection{Example to Section 2.3$_{mf}$}
	\begin{exa}\label{ex2}
		In this example we present a colored space $\Zsp_\Gamma \subset \mathrm{Sym}(16;\real)$, which has a component
		$\mathrm{Herm}(2; \quat)$.
		Let $\Gamma=\left<\sigma_1,\sigma_2\right>$ be the subgroup of $\mathfrak{S}_{16}$ generated by the two permutations
		\begin{align*}
			\sigma_1= (1, 2, 5, 6)(3, 4, 7, 8)(9, 10, 13, 14)(11, 12, 15, 16),\\
			\sigma_2=(1, 3, 5, 7)(2, 8, 6, 4)(9, 11, 13, 15)(10, 16, 14, 12).
		\end{align*}
		The space $\mathcal{Z}_\Gamma$ consists of matrices of the form
		\[ 
		X=\left(
		\begin{array}{cccccccccccccccc}
			\alpha_1 & \alpha_2 & \alpha_3 & \alpha_4 & \alpha_5 & \alpha_2 & \alpha_3 & \alpha_4 & \gamma_1 & \gamma_2 & \gamma_3 & \gamma_4 & \gamma_5 & \gamma_6 & \gamma_7 & \gamma_8 \\
			\alpha_2 & \alpha_1 & \alpha_4 & \alpha_3 & \alpha_2 & \alpha_5 & \alpha_4 & \alpha_3 & \gamma_6 & \gamma_1 & \gamma_8 & \gamma_3 & \gamma_2 & \gamma_5 & \gamma_4 & \gamma_7 \\
			\alpha_3 & \alpha_4 & \alpha_1 & \alpha_2 & \alpha_3 & \alpha_4 & \alpha_5 & \alpha_2 & \gamma_7 & \gamma_4 & \gamma_1 & \gamma_6 & \gamma_3 & \gamma_8 & \gamma_5 & \gamma_2 \\
			\alpha_4 & \alpha_3 & \alpha_2 & \alpha_1 & \alpha_4 & \alpha_3 & \alpha_2 & \alpha_5 & \gamma_8 & \gamma_7 & \gamma_2 & \gamma_1 & \gamma_4 & \gamma_3 & \gamma_6 & \gamma_5 \\
			\alpha_5 & \alpha_2 & \alpha_3 & \alpha_4 & \alpha_1 & \alpha_2 & \alpha_3 & \alpha_4 & \gamma_5 & \gamma_6 & \gamma_7 & \gamma_8 & \gamma_1 & \gamma_2 & \gamma_3 & \gamma_4 \\
			\alpha_2 & \alpha_5 & \alpha_4 & \alpha_3 & \alpha_2 & \alpha_1 & \alpha_4 & \alpha_3 & \gamma_2 & \gamma_5 & \gamma_4 & \gamma_7 & \gamma_6 & \gamma_1 & \gamma_8 & \gamma_3 \\
			\alpha_3 & \alpha_4 & \alpha_5 & \alpha_2 & \alpha_3 & \alpha_4 & \alpha_1 & \alpha_2 & \gamma_3 & \gamma_8 & \gamma_5 & \gamma_2 & \gamma_7 & \gamma_4 & \gamma_1 & \gamma_6 \\
			\alpha_4 & \alpha_3 & \alpha_2 & \alpha_5 & \alpha_4 & \alpha_3 & \alpha_2 & \alpha_1 & \gamma_4 & \gamma_3 & \gamma_6 & \gamma_5 & \gamma_8 & \gamma_7 & \gamma_2 & \gamma_1 \\
			\gamma_1 & \gamma_6 & \gamma_7 & \gamma_8 & \gamma_5 & \gamma_2 & \gamma_3 & \gamma_4 & \beta_1 & \beta_2 & \beta_3 & \beta_4 & \beta_5 & \beta_2 & \beta_3 & \beta_4 \\
			\gamma_2 & \gamma_1 & \gamma_4 & \gamma_7 & \gamma_6 & \gamma_5 & \gamma_8 & \gamma_3 & \beta_2 & \beta_1 & \beta_4 & \beta_3 & \beta_2 & \beta_5 & \beta_4 & \beta_3 \\
			\gamma_3 & \gamma_8 & \gamma_1 & \gamma_2 & \gamma_7 & \gamma_4 & \gamma_5 & \gamma_6 & \beta_3 & \beta_4 & \beta_1 & \beta_2 & \beta_3 & \beta_4 & \beta_5 & \beta_2 \\
			\gamma_4 & \gamma_3 & \gamma_6 & \gamma_1 & \gamma_8 & \gamma_7 & \gamma_2 & \gamma_5 & \beta_4 & \beta_3 & \beta_2 & \beta_1 & \beta_4 & \beta_3 & \beta_2 & \beta_5 \\
			\gamma_5 & \gamma_2 & \gamma_3 & \gamma_4 & \gamma_1 & \gamma_6 & \gamma_7 & \gamma_8 & \beta_5 & \beta_2 & \beta_3 & \beta_4 & \beta_1 & \beta_2 & \beta_3 & \beta_4 \\
			\gamma_6 & \gamma_5 & \gamma_8 & \gamma_3 & \gamma_2 & \gamma_1 & \gamma_4 & \gamma_7 & \beta_2 & \beta_5 & \beta_4 & \beta_3 & \beta_2 & \beta_1 & \beta_4 & \beta_3 \\
			\gamma_7 & \gamma_4 & \gamma_5 & \gamma_6 & \gamma_3 & \gamma_8 & \gamma_1 & \gamma_2 & \beta_3 & \beta_4 & \beta_5 & \beta_2 & \beta_3 & \beta_4 & \beta_1 & \beta_2 \\
			\gamma_8 & \gamma_7 & \gamma_2 & \gamma_5 & \gamma_4 & \gamma_3 & \gamma_6 & \gamma_1 & \beta_4 & \beta_3 & \beta_2 & \beta_5 & \beta_4 & \beta_3 & \beta_2 & \beta_1 
		\end{array}
		\right).
		\]
		The irreducible factorization of the determinant is given by
		{\tiny
			\begin{align*}
				\Det{X} &= \left((\gamma_1-\gamma_5)^2+(\gamma_2-\gamma_6)^2+(\gamma_3-\gamma_7)^2+(\gamma_4-\gamma_8)^2 -(\alpha_1-\alpha_5)(\beta_1-\beta_5)\right)^4 \\
				&\cdot\left((\alpha_1 - 2 (\alpha_2 + \alpha_3 - \alpha_4) + \alpha_5) (\beta_1 - 2 (\beta_2 + \beta_3 - \beta_4) + 
				\beta_5) - (\gamma_1 - \gamma_2 - \gamma_3 + \gamma_4 + \gamma_5 - \gamma_6 - \gamma_7 + \gamma_8)^2\right) \\
				&\cdot\left((\alpha_1 - 2 (\alpha_2 - \alpha_3 + \alpha_4) + \alpha_5) (\beta_1 - 2 (\beta_2 - \beta_3 + \beta_4) + 
				\beta_5) - (\gamma_1 - \gamma_2 + \gamma_3 - \gamma_4 + \gamma_5 - \gamma_6 + \gamma_7 - \gamma_8)^2\right) \\
				&\cdot\left(
				(\alpha_1 + 2 (\alpha_2 - \alpha_3 - \alpha_4) + \alpha_5) (\beta_1 + 2 (\beta_2 - \beta_3 - \beta_4) + 
				\beta_5) - (\gamma_1 + \gamma_2 - \gamma_3 - \gamma_4 + \gamma_5 + \gamma_6 - \gamma_7 - \gamma_8)^2
				\right) \\
				& \cdot\left(
				(\alpha_1 + 2 (\alpha_2 + \alpha_3 + \alpha_4) + \alpha_5) (\beta_1 + 2 (\beta_2 + \beta_3 + \beta_4) + 
				\beta_5) - (\gamma_1 + \gamma_2 + \gamma_3 + \gamma_4 + \gamma_5 + \gamma_6 + \gamma_7 + \gamma_8)^2
				\right).
			\end{align*}
		}\noindent Thus, Lemma 4$_{mf}$ gives us that $L=5$ and
		\begin{align*}
			r  = (2,2,2,2,2), \quad
			k  = (4,1,1,1,1), \quad 
			d  = (4,1,1,1,1).
		\end{align*}
		This in turn implies
		\[ 
		\mathcal{Z}_\Gamma \,{\simeq}\, \mathrm{Herm}(2; \mathbb{H}) \oplus \mathrm{Sym}(2;\real)^{\oplus 4}.
		\]
		As a matter of fact,
		the group $\Gamma$  has four $1$-dimensional representations and one $4$-dimensional real irreducible representation,
		and each representation appears twice in $\real^{16}$.
	\end{exa}	
	
	\subsection{Proofs of Theorem 5$_{mf}$ and Theorem 6$_{mf}$}
	\begin{proof}[Proof of Theorem 5$_{mf}$ and Theorem 6$_{mf}$]
		Let $M:=\left\lfloor\frac{N}{2}\right\rfloor$ and denote the  irreducible representations of $\Gamma$ by
		\begin{align*}
			B_0 &\colon \Gamma \owns \sigma^k \mapsto 1 \in \mathrm{GL}(1; \real),\\
			B_\alpha &\colon \Gamma \owns \sigma^k \mapsto \mathrm{Rot}\left(\frac{2\pi \alpha k}{N}\right) \in \mathrm{GL}(2; \real)
			\quad (1 \le \alpha < N/2),\\
			B_{N/2} &\colon \Gamma \owns \sigma^k \mapsto (-1)^k \in \mathrm{GL}(1;\real)
			\quad (\mbox{if }N \mbox{ is even}),
		\end{align*}
		where 
		$\mathrm{Rot}(\theta)$
		denotes the rotation matrix  
		$\begin{pmatrix} \cos \theta & - \sin \theta \\ \sin \theta & \cos \theta \end{pmatrix}$
		for $\theta \in \real$.
		Then all the equivalence classes of the irreducible representations of $\Gamma$ are	
		$[B_0],\, [B_1], \dots, [B_M]$ 
		whether $N = 2M$ or $N = 2M+1$.
		We have
		$k_\alpha = d_\alpha = \begin{cases} 1 & (\alpha = 0 \mbox{ or }N/2), \\ 2 & \mbox{(otherwise)}. \end{cases}$
		
		Recall that $\{i_1, \dots, i_C\}$ is a complete system of representatives of the $\Gamma$-orbits,
		and, for each $c=1, \dots, C$, 
		$p_c$ is the cardinality of the $\Gamma$-orbit through $i_c$.
		Let $\zeta_c := \exp(2\pi \sqrt{-1}/p_c)$.
		When $1 \le \beta < p_c/2$, 
		we have 
		\[
		v^{(c)}_{2\beta} + \sqrt{-1} v^{(c)}_{2 \beta +1} 
		= \sqrt{\frac{2}{p_c}} \sum_{k=0}^{p_c-1} \zeta_c^{\beta k} e_{\sigma^k(i_c)}.
		\]
		Thus
		\begin{align*}
			R&(\sigma)( v^{(c)}_{2\beta} + \sqrt{-1} v^{(c)}_{2 \beta +1})
			= \sqrt{\frac{2}{p_c}} \sum_{k=0}^{p_c-1} \zeta_c^{\beta k} e_{\sigma^{k+1}(i_c)}
			= \sqrt{\frac{2}{p_c}} \sum_{k=0}^{p_c-1} \zeta_c^{\beta (k-1)} e_{\sigma^{k}(i_c)}
			\\
			&= \zeta_c^{-\beta}( v^{(c)}_{2\beta} + \sqrt{-1} v^{(c)}_{2 \beta +1})\\
			&= \Bigl\{ \cos \Bigl( \frac{2\pi \beta}{p_c}\Bigr) v^{(c)}_{2\beta} 
			+ \sin \Bigl( \frac{2\pi \beta}{p_c}\Bigr) v^{(c)}_{2\beta +1} \Bigr\} 
			+\sqrt{-1} \Bigl\{ -\sin \Bigl( \frac{2\pi \beta}{p_c}\Bigr) v^{(c)}_{2\beta}
			+ \cos \Bigl( \frac{2\pi \beta}{p_c}\Bigr) v^{(c)}_{2\beta+1} \Bigr\},   
		\end{align*}
		where we have used $\sigma^{p_c}(i_c) = i_c$ and $\zeta_c^{p_c}=1$ at the second equality.
		It follows that 
		\begin{equation} \label{eqn:alpha-beta}
			R(\sigma) \begin{pmatrix} v^{(c)}_{2\beta} & v^{(c)}_{2 \beta +1} \end{pmatrix}
			= \begin{pmatrix} v^{(c)}_{2\beta} & v^{(c)}_{2 \beta +1} \end{pmatrix} \mathrm{Rot}\Bigl( \frac{2\pi \beta}{p_c} \Bigr)
			= \begin{pmatrix} v^{(c)}_{2\beta} & v^{(c)}_{2 \beta +1} \end{pmatrix} B_{\alpha}(\sigma)
		\end{equation}
		with ${\beta}/{p_c} = {\alpha}/{N}$.
		Similarly, we have
		\begin{align*}
			R(\sigma) v^{(c)}_1 &= v^{(c)}_1 = B_0(\sigma) v^{(c)}_1, \\
			R(\sigma) v^{(c)}_{p_c} &= -v^{(c)}_{p_c} = B_{N/2}(\sigma) v^{(c)}_{p_c}
			\quad (\mbox{if }p_c \mbox{ and }N \mbox{ are even}).
		\end{align*}
		Therefore, for $\alpha = 0, \dots, \lfloor N/2\rfloor$,
		the multiplicity $r_\alpha$ of the representation $B_\alpha$ of $\Gamma$ in $(R, \real^p)$ is equal to the number of $c$
		such that ${\beta}/{p_c} = {\alpha}/{N}$ with some $\beta\in\mathbb{N}$.
		In other words,
		\begin{equation} \label{eqn:r_alpha}
			r_\alpha = \#\set{c}{\alpha p_c  \mbox{ is a multiple of }N} \quad (0 \le \alpha \le \lfloor N/2\rfloor).
		\end{equation}    
		Then we have
		\[
		\Zsp_\Gamma \simeq  
		\bigoplus_{r_\alpha>0} \mathrm{Herm}(r_\alpha; \mathbb{K}_\alpha). 
		\]
	\end{proof}
	
	\subsection{Finding structure constants when $\Gamma$ is not cyclic}
	In Section 2.3$_{mf}$, we gave  a general algorithm for determining the  structure constants as well as the invariant measure $\varphi_\Gamma$. In principle, the factorization of a determinant can be done e.g. in \textsc{Python}, however there are some limitations regarding the dimension of a matrix. If the $p\times p$ matrix is not sparse, then the number of terms in the usual Laplace expansion of a determinant produces a polynomial with $p!$ terms. The RAM memory requirements for calculating such a polynomial would be in excess of $p!$, which cannot be handled on a standard PC even for moderate $p$. Depending on the subgroup and the method of calculating the determinant, we were able to obtain the determinant for models of dimensions up to $10$-$20$. 
	In order to factorize the determinant for moderate to high dimensions, we want to find an orthogonal  matrix $U$ such that $\transp{U}\cdot X\cdot U$ is sparse enough for a computer to calculate its determinant $\Det{\transp{U}\cdot X\cdot U}=\Det{X}$. The matrix $U_\Gamma$ from \eqref{genform} is in general very hard to obtain, but we propose an easy surrogate.
	
	\begin{pro}\label{pro:0}
		Assume that $\Gamma$ be a subgroup of $\mathfrak{S}_p$. Take $\sigma_0 \in \mathfrak{S}_p$
		for which the cyclic group $\Gamma_0 := \langle \sigma_0 \rangle$ generated by $\sigma_0$
		has the same orbits in $V$ as $\Gamma$. 
		Then, for any $X\in\mathcal{Z}_\Gamma$ one has
		\begin{align}\label{eq:xy}
			\transp{U_{\Gamma_0}} \cdot X \cdot U_{\Gamma_0} = \begin{pmatrix} x_1 & 0 \\ 0 & y \end{pmatrix}
		\end{align}
		with $x_1 \in \mathrm{Sym}(C; \real)$
		and $y \in \mathrm{Sym}(p-C; \real)$, where $C$ is the number of cycles of $\sigma_0$. Matrix $U_{\Gamma_0}$ is the orthogonal matrix constructed in Theorem 6$_{mf}$ for the cyclic subgroup $\Gamma_0$.		
	\end{pro}
	\begin{proof}
		
		We observe that a vector $v \in \real^p$ is $\Gamma$-invariant
		(i.e. $R(\sigma) v = v$ for all $\sigma \in \Gamma$) 
		if and only if $R(\sigma_0) v = v$ if and only if $v$ is constant on $\Gamma$-orbits (i.e. $v_i = v_j$ if $i$ and $j$ belong to the same orbit of $\Gamma$).
		The first $C$ column vectors of $U_{\Gamma_0}$ are 
		$v^{(1)}_1, v^{(2)}_1, \dots, v^{(C)}_1$, which are $\Gamma$-invariant.
		The space $V_1 := \mathrm{span}\set{ v^{(c)}_1 }{c=1, \dots, C} \subset \real^p$ 
		is the trivial-representation-component of $\Gamma$ as explained after \eqref{eqn:Balpha_decomp}.
		Therefore, 
		if $\alpha_1$ is the trivial representation of 
		$\Gamma$, then $r_1 = C$ and $d_1 = k_1 = 1$.
		
		The orthogonal complement $V_1^{\perp}$ of $V_1$ is spanned by the rest of
		$v^{(c)}_{\beta}$, $1 \le c \le C$, $1 < \beta \le 2\lfloor p_c/2\rfloor$.
		For $X \in \Zsp_{\Gamma}$, 
		we see that $X\cdot v \in V_1$ for $v \in V_1$ 
		and that $X\cdot w \in V_1^{\perp}$ for $w \in V_1^{\perp}$.
		In this way we obtain \eqref{eq:xy}.
	\end{proof}
	
	We note that in general, there are no inclusion relations between groups $\Gamma$ and $\Gamma_0$. Moreover, the correspondence $\phi_1\colon \Zsp_{\Gamma} \owns X \mapsto x_1 \in \mathrm{Sym}(C; \real)$
	is exactly the Jordan algebra homomorphism 
	defined before Corollary 3$_{mf}$.
	By Proposition \ref{pro:0} we obtain
	\begin{align}\label{eq:Detxy}
		\Det{X} = \Det{x_1}\Det{y},
	\end{align}
	while the factor $\Det{x_1} = \det(\phi_1(X))$ 
	is an irreducible polynomial of degree $r_1 = C$. In this way, for any subgroup $\Gamma$, we are able to factor out the polynomial of degree equal to the number of $\Gamma$-orbits in $V$ easily. On the other hand, 
	the factorization of $\Det{y}$ requires 
	study of the subrepresentation $R$ of $\Gamma$ on $V_1^{\perp}$, 
	where the group $\Gamma_0$ is useless in general.

	\begin{exa}\label{ex:ex4}
		Let $\Gamma=\langle (1,2,3), \, (4,5,6)\rangle \subset \mathfrak{S}_6$, which is not a cyclic group.
		The space $\Zsp_\Gamma$ consists of symmetric matrices of the form
		\[
		X= \begin{pmatrix} a & b & b & e & e & e \\ b & a & b & e & e & e \\ b & b & a & e & e & e \\
			e & e & e & c & d & d \\ e & e & e & d & c & d \\ e & e & e & d & d & c \end{pmatrix}
		\]
		and moreover, $\Zsp_{\Gamma}$ does not coincide with $\Zsp_{\left<\sigma\right>}$ for any $\sigma\in\mathfrak{S}_6$.
		Noting that the group
		$\Gamma$ has two orbits: $\{1,2,3\}$ and $\{4,5,6 \}$,
		we define 
		$\sigma_0 := (1,2,3)(4,5,6)$.
		Taking $i_1 = 1$ and $i_2 = 4$, 
		we have
		\[
		U_{\Gamma_0}
		= \begin{pmatrix}
			1/ \sqrt{3} & 0 &  \sqrt{2/3} & 0 & 0 & 0\\
			1/\sqrt{3} & 0 & -1/ \sqrt{6} & 1/\sqrt{2} & 0 & 0 \\
			1/\sqrt{3} & 0 & -1/ \sqrt{6} &-1/\sqrt{2} & 0 & 0 \\
			0 & 1/ \sqrt{3} & 0 & 0 &  \sqrt{2/3} & 0 \\
			0 & 1/ \sqrt{3} & 0 & 0 &  -1/\sqrt{6} & 1/\sqrt{2} \\
			0 & 1/ \sqrt{3} & 0 & 0 &  -1/\sqrt{6} &-1/\sqrt{2} 
		\end{pmatrix}.
		\]
		Note that the first two column vectors
		$\transp{\begin{pmatrix} 1/ \sqrt{3}, & 1/ \sqrt{3},  & 1/ \sqrt{3}, & 0, & 0, & 0 \end{pmatrix} }$
		and\\
		$\transp{ \begin{pmatrix} 0, & 0, & 0, & 1/ \sqrt{3}, & 1/ \sqrt{3},  & 1/ \sqrt{3} \end{pmatrix} }$
		of $U_{\Gamma_0}$ are $\Gamma$-invariant.
		By direct calculation we verify that 
		$\transp{U_{\Gamma_0}} \cdot X \cdot U_{\Gamma_0}$ is of the form
		\[
		\begin{pmatrix} 
			A & B & 0 & 0 & 0 & 0 \\
			B & C & 0 & 0 & 0 & 0 \\
			0 & 0 & D & 0 & 0 & 0 \\
			0 & 0 & 0 & D & 0 & 0 \\
			0 & 0 & 0 & 0 & E & 0 \\
			0 & 0 & 0 & 0 & 0 & E 
		\end{pmatrix},
		\]
		where $A,B, \cdots, E$ are linear functions of $a,b, \cdots, e$.
		The matrices $x_1$ and $y$ are
		$\begin{pmatrix} A & B \\ B & C \end{pmatrix}$
		and 
		$\begin{pmatrix} D & 0 & 0 & 0 \\ 0 & D & 0 &  0 \\ 0 & 0 & E & 0 \\ 0 & 0 & 0 & E \end{pmatrix}$
		respectively. The matrix $y$ is of such simple form, because $\Gamma_0$ is a subgroup of $\Gamma$ in this case.
	\end{exa}
	
	We cannot expect that the matrix $y$ in \eqref{eq:xy} is always of a nice form, as in the example above. 
	However, we note that in many examples we considered, the matrix $y$ was sparse, which also makes the problem of calculating $\Det{X}$ much more feasible on a standard PC.
	
	In general $\Gamma_0$ defined above is not a subgroup of $\Gamma$. As we argue below, valuable insight about  the factorization of $\Zsp_{\Gamma}$ can be obtained by studying 
	cyclic subgroups of $\Gamma$. The argument is based on the following easy result.
	\begin{lem}\label{lem:incl}
		Let $\Gamma_1$ be a subgroup of $\Gamma$. Then $\Zsp_{\Gamma}\subset \Zsp_{\Gamma_1}$.
	\end{lem}
	Let $\Gamma_1$ be a cyclic subgroup of $\Gamma$. Then using Theorem 5$_{mf}$, we can easily calculate structure constant corresponding to $\Gamma_1$. Let $U_{\Gamma_1}$ be the orthogonal matrix constructed in Theorem 6$_{mf}$. By Lemma \ref{lem:incl} and \eqref{genform}, for any $X\in\Zsp_{\Gamma}$ the matrix $\transp{U_{\Gamma_1}} \cdot X \cdot U_{\Gamma_1}$ is of the form
	\[
	\begin{pmatrix} 
		M_{\mathbb{K}_1}(x_1^\prime) \otimes I_{\tfrac{k_1}{d_1}} & & & \\
		& M_{\mathbb{K}_2}(x_2^\prime) \otimes I_{\tfrac{k_2}{d_2}} & & \\
		& & \ddots & \\
		& & & M_{\mathbb{K}_L}(x_L^\prime ) \otimes I_{\tfrac{k_L}{d_L}} 
	\end{pmatrix},
	\]
	where $x_i^\prime \in \mathrm{Herm}(r_i; \K_i)$, $i=1,\ldots,L$.
	In particular, we have $k_1=d_1=1$ and $r_1$ is the number of $\Gamma_1$-orbits in $\{1,\ldots,p\}$. Thus, we have 
	$M_{\mathbb{K}_1}(x_1^\prime) \otimes I_{k_1/d_1}  = x_1^\prime\in \mathrm{Sym}(r_1;\real)$.
	In contrast to \eqref{eq:xy}, $x_1^\prime$ in general can be further factorized
	and we know that $\Det{x_1}$ from \eqref{eq:Detxy} is an irreducible factor 
	of $\Det{x_1^\prime}$.
	In conclusion,
	each cyclic subgroup of  the general group $\Gamma$ brings various information about the factorization.
	
	\begin{exa}
		We continue Example \ref{ex:ex4}. Let $\Gamma_1 = 
		\langle (1,2,3) \rangle$, which is a subgroup of $\Gamma$.
		There are four $\Gamma_1$-orbits in $V$, that is, $\{1,2,3\}$, $\{4\}$, $\{5\}$, and $\{6\}$.
		We have
		\[
		U_{\Gamma_1} = \begin{pmatrix}
			1/\sqrt{3} & 0 & 0 & 0 & \sqrt{2/3} & 0 \\
			1/\sqrt{3} & 0 & 0 & 0 &-1/\sqrt{6} & 1/\sqrt{2} \\
			1/\sqrt{3} & 0 & 0 & 0 &-1/\sqrt{6} &-1/\sqrt{2} \\
			0              & 1 & 0 & 0 & 0 & 0 \\
			0              & 0 & 1 & 0 & 0 & 0 \\
			0              & 0 & 0 & 1 & 0 & 0 
		\end{pmatrix}. 
		\]
		For $X \in Z_{\Gamma}$, 
		we see that
		$\transp{U_{\Gamma_1}} \cdot X \cdot U_{\Gamma_1}$ is of the form
		\[\begin{pmatrix}
			A_{11} & A_{21} & A_{31} & A_{41} & 0 & 0 \\
			A_{21} & A_{22} & A_{32} & A_{42} & 0 & 0 \\
			A_{31} & A_{32} & A_{33} & A_{43} & 0 & 0 \\
			A_{41} & A_{42} & A_{43} & A_{44} & 0 & 0 \\
			0 & 0 & 0 & 0 & D & 0 \\
			0 & 0 & 0 & 0 & 0 & D 
		\end{pmatrix}, 
		\]
		where $A_{ij}$ are linear functions of $a,b, \dots, e$, 
		but they are not linearly independent.  
		Indeed, we have
		\[ 
		\mathrm{Det} \begin{pmatrix}
			A_{11} & A_{21} & A_{31} & A_{41}\\
			A_{21} & A_{22} & A_{32} & A_{42}\\
			A_{31} & A_{32} & A_{33} & A_{43}\\
			A_{41} & A_{42} & A_{43} & A_{44} \end{pmatrix}
		= E^2 \det \begin{pmatrix} A & B \\ B & C \end{pmatrix},
		\]
		which exemplifies the fact that $\Det{x_1}$ is an irreducible factor of $\Det{x_1^\prime}$.
	\end{exa}

	\subsection{Gamma integrals}
	In this section we prove Proposition 7$_{mf}$, Theorem 8$_{mf}$ and Theorem 9$_{mf}$. Proofs of these results are based on Lemma \ref{lem:integral} below.
	
	The key ingredient to compute the Gamma integral on $\Pcone_\Gamma$ is the block decomposition of $\mathcal{Z}_\Gamma$. 	We assume that $\mathcal{Z}_\Gamma$ is in the form of \eqref{genform}.
	Let $\Omega_i$ denote the symmetric cone of the simple Jordan algebra $\A_i =\mathrm{Herm}(r_i; \K_i)$ and $d_i=\dim_{\real} \K_i$, $i=1,\ldots,L$. 
	We have $\dim\Omega_i= r_i + r_i(r_i-1)d_i/2$. 
	Recall that, for $X\in \Zsp_\Gamma$ represented as in \eqref{genform}, we
	write $\phi_i(X)= x_i \in \A_i$ for $i=1, \dots, L$.
	
	\begin{lem}\label{lem:integral}
		For any $Y\in\Pcone_\Gamma$ and $\lambda_i>-1$, $i=1,\ldots,L$, we have 
		\begin{align}\label{eq:integral}
			\int_{\Pcone_\Gamma} \prod_{i=1}^L \det(\phi_i(X))^{\lambda_i} e^{-\Tr{Y\cdot X}} \dd X = 
			e^{-B_\Gamma} \left( \prod_{i=1}^L  k_i^{-r_i\lambda_i }\right) \prod_{i=1}^L  \frac{\Gamma_{\Omega_i}(\lambda_i+\dim\Omega_i/r_i)}{ \det(\phi_i(Y))^{\lambda_i+\dim\Omega_i/r_i}},
		\end{align}
		where $B_\Gamma$ is defined in $(15)_{mf}$.
	\end{lem}
	
	\begin{proof}[Proof of Lemma \ref{lem:integral}]
		Denote the left hand side of \eqref{eq:integral} by $I$. Let us change variables $x_i=\phi_i(X)$ for $i=1,\ldots,L$.
		By $(10)_{mf}$ and $(14)_{mf}$ we obtain
		\begin{align*}
			I &= e^{B_\Gamma}  \prod_{i=1}^L \int_{\Omega_i} \det(x_i)^{\lambda_i} e^{- k_i \mathrm{tr}[\phi_i(Y)\cdot x_i]}  m_i(\mathrm{d}x_i).
			\intertext{Each integral can be calculated using $(13)_{mf}$ for $\lambda_i>-1$ and $\phi_i(Y)\in\Omega_i$, $i=1,\ldots,L$. Hence,}
			I& = e^{B_\Gamma}  \prod_{i=1}^L \Gamma_{\Omega_i}(\lambda_i+\dim\Omega_i/r_i) \det\left(k_i\, \phi_i(Y)\right)^{-\lambda_i-\dim\Omega_i/r_i} \\
			&= e^{B_\Gamma}  \left( \prod_{i=1}^L  k_i^{-r_i\lambda_i - \dim \Omega_i }\right)  \prod_{i=1}^L\frac{\Gamma_{\Omega_i}(\lambda_i+\dim\Omega_i/r_i)}{ \det(\phi_i(Y))^{\lambda_i+\dim\Omega_i/r_i}} \\
			&= e^{-B_\Gamma}  \left( \prod_{i=1}^L  k_i^{-r_i\lambda_i}\right) \prod_{i=1}^L \frac{\Gamma_{\Omega_i}(\lambda_i+\dim\Omega_i/r_i)}{ \det(\phi_i(Y))^{\lambda_i+\dim\Omega_i/r_i}} 
		\end{align*}
		and so we obtain \eqref{eq:integral}.
	\end{proof}
	
	\begin{proof}[Proof of Proposition 7$_{mf}$]
		Recall that $\Pcone_{\Gamma}$ is a symmetric cone, so that it coincides with its dual cone, $\Pcone_{\Gamma}^\ast$.  Thus,
		\[
		\varphi_\Gamma(Y) =  e^{B_\Gamma} { \left( \prod_{i=1}^L \frac{1}{ \Gamma_{\Omega_i}(\dim \Omega_i /r_i)}\right) }
		\int_{\Pcone_{\Gamma}} e^{-\Tr{Y\cdot X}}\dd X\qquad (Y\in \Pcone_\Gamma).
		\]
		Setting $\lambda_1=\ldots=\lambda_L=0$ in \eqref{eq:integral} we obtain the expression for $\varphi_\Gamma(Y)$.
	\end{proof}
	
	\begin{proof}[Proof of Theorem 8$_{mf}$ and Theorem 9$_{mf}$]
		Recall that for $X\in\Pcone_\Gamma$ we have $\Det{X}=\prod_{i=1}^L \det(\phi_i(X))^{k_i}$, 
		where the map
		$\phi_i: \Zsp_\Gamma \to \mathrm{Herm} (r_i; \K_i)$ is a Jordan algebra homomorphism, $i=1,\ldots,L$.
		
		If $\lambda_i=k_i\lambda-\dim\Omega_i/r_i$ with
		$\lambda > \max_{i=1,\ldots,L} \left\{ (r_i-1)d_i/(2 k_i)\right\}$,
		then \eqref{eq:integral} implies
		\[
		\prod_{i=1}^L  k_i^{-r_i\lambda_i} = e^{-A_\Gamma \lambda+2B_\Gamma}
		\quad\mbox{and}\quad 
		\prod_{i=1}^L \det(\phi_i(Y))^{-\lambda_i-\dim\Omega_i/r_i} = \left(\prod_{i=1}^L \det(\phi_i(Y))^{k_i}\right)^{-\lambda}
		\]
		
		If $\lambda_i=k_i\lambda$ with 	$\lambda > \max_{i=1,\ldots,L} \left\{-1/k_i\right\}$, 
		then by \eqref{eq:integral} we obtain
		\[
		\prod_{i=1}^L  k_i^{-r_i\lambda_i} = e^{-A_\Gamma \lambda}
		\quad\mbox{and}\quad 
		\prod_{i=1}^L \det(\phi_i(Y))^{-\lambda_i-\dim\Omega_i/r_i} = \left(\prod_{i=1}^L \det(\phi_i(Y))^{k_i}\right)^{-\lambda}\varphi_\Gamma(Y).
		\]
	\end{proof}
	
	\subsection{Jacobian}
	\begin{proof}[Proof of Lemma 13$_{mf}$]
		First observe that
		\begin{align*}
			(X+h)^{-1}-X^{-1}  = (X+h)^{-1}\cdot \left[X - (X+h) \right] \cdot X^{-1}
			= - X^{-1}\cdot h\cdot X^{-1} +o(h),
		\end{align*}
		so that, 
		the Jacobian of $\Pcone_{\Gamma}\ni X\mapsto X^{-1}\in \Pcone_{\Gamma}$ equals $\mathrm{Det}_{\mathrm{End}}( \mathbb{P}_{X^{-1}})$, where $\mathrm{Det}_{\mathrm{End}}$ is the determinant in the space of endomorphisms of $\Zsp_{\Gamma}$ and for any $X\in \Zsp_{\Gamma}$ by $\mathbb{P}_{X}$  we denote the linear map on $\Zsp_{\Gamma}$ to itself defined by $\mathbb{P}_{X} Y = X\cdot Y\cdot X$. It is easy to see that for any $X\in\Pcone_{\Gamma}$ we have $\mathbb{P}_{X}\in \mathrm{G}(\Pcone_{\Gamma})$. Indeed, since $\mathbb{P}_X Y$ is positive definite for $Y\in\Pcone_{\Gamma}$, it is enough to verify that 
		\[ 
		R(\sigma)\cdot \left[ \mathbb{P}_XY \right] = \left[ \mathbb{P}_XY \right]\cdot R(\sigma)\qquad (\sigma\in\Gamma).
		\]
		This follows quickly by the fact that $X, Y\in\Pcone_{\Gamma}$. 
		Further, by the $\mathrm{G}(\Pcone_{\Gamma})$ invariance of $\varphi_\Gamma$, we have
		\[
		\varphi_{\Gamma}( g X) = \left| \mathrm{Det}_{\mathrm{End}} (g) \right|^{-1} \varphi_{\Gamma}(X)\qquad (g\in \mathrm{G}(\Pcone_{\Gamma})).
		\]
		Taking $g=\mathbb{P}_{X^{-1}}$, we eventually obtain
		\[
		\mathrm{Det}_{\mathrm{End}}( \mathbb{P}_{X^{-1}}) = \frac{ \varphi_{\Gamma}(X)}{\varphi_{\Gamma}(X^{-1})} = \left[\varphi_{\Gamma}(X)\right]^2,
		\]
		where the latter inequality can be easily verified by $(16)_{mf}$.
	\end{proof}

	\section{Additions to Section 6$_{mf}$}
	\subsection{Real data example}
	We applied our procedure to the breast cancer data set considered in Section 7 of \cite{HL08}, see also \cite{Miller}. Following approach of H{\o}jsgaard and Lauritzen, we consider set of $p=150$ genes (designations of these genes can be read from \cite[Fig. 11]{HL08}) and $n=58$ samples with mutation in the p53 sequence. We numbered the variables alphabetically. Since $p>n$, only parsimonious models can be fitted at all. 
	
	In \cite{HL08}, the variables were standardized to have zero mean and unit variance. As the authors write, due to ``an issue of scaling of the variables'', model selection within RCOR models (a superclass of RCOP models) was performed. However, when the search is done among RCOP models, the scaling ensuring unit variances favors transitive subgroups. Recall that a cyclic subgroup is transitive if it is generated by a permutation consisting of one big cycle. Therefore we only centered the data.
	
	We run the Metropolis-Hastings algorithm (Algorithm 13$_{mf}$) with hyper-parameters $D=I_p$ and $\delta=3$ for $150\,000$ iterations.  
	Cardinality of the search space is not easy to compute, but already for $p=130$, the number of cyclic subgroups is of magnitude $10^{217}$, see OEIS sequence A051625. The cyclic subgroup with highest estimated posterior probability ($7.1\%$) is given by $\hat{\Gamma}=\left\langle \sigma^\ast\right\rangle$, where
	{\tiny \begin{align*}
			\sigma^\ast=&(1, 2, 139, 149, 61, 52, 8, 145)
			(3, 11, 9, 89, 6, 102, 120, 4)
			(5, 47, 90)
			(7,
			13,
			138,
			91,
			117,
			142,
			143,
			72,
			146,
			50,
			136,
			22,
			57,
			87,
			124,\\
			&
			114,
			84,
			30)
			(10, 99, 39, 21, 101, 26, 37, 73)
			(12, 77, 100, 133, 122)
			(14, 19, 76, 147)
			(15, 71, 127, 110)
			(16, 92, 83, 34, 140, 27, 49, 137)\\
			&
			(17, 98, 69)
			(18, 65, 134, 88, 107, 75, 108, 106, 82, 109, 123, 68)
			(20, 51, 135, 105, 38, 96, 25, 45)
			(23, 111, 24, 42, 67, 43, 131, 112)\\
			&		(31, 58, 66, 94, 81)
			(32, 33)
			(35, 93, 64, 86, 128, 148, 132, 103, 60, 150, 144, 129, 118, 70, 97, 121)
			(36, 85, 141)
			(44, 56, 119, 126, 104, \\
			&78, 79, 48)
			(46, 130, 115, 74, 116, 59, 113, 125, 95).
	\end{align*}}
	The order of $\hat{\Gamma}$ is $720$. 
	The structure constants of $\Zsp_{\hat{\Gamma}}$ are $L=21$,
	\begin{align*}
		r=(29,  1,  1,  1,  2,  8,  2,  1,  2,  2, 11,  1,  1,  6,  8,  1,  2,1,  1,  2, 13),\\
		d=k=(1, 2, 2, 2, 2, 2, 2, 2, 2, 2, 2, 2, 2, 2, 2, 2, 2, 2, 2, 2, 1),
	\end{align*}
	which imply that  $\dim\Zsp_{\hat{\Gamma}}=844$. Although the number of parameters (colors) of $\Zsp_{\hat{\Gamma}}$ is rather high, the MLE of $\Sigma$ exists in this model. Indeed, in view of Corollary 12$_{mf}$, we have
	\[
	n_0= \max_{i=1,\ldots,L}\left\{\frac{r_i d_i}{k_i}\right\} = 29
	\]
	and so $(23)_{mf}$ is satisfied. 
	
	The color pattern of the space of $p\times p$ matrices from $\mathcal{Z}_{\hat{\Gamma}}$ is depicted in Fig. \ref{pic1} (a). Entries which correspond to the same color in a figure are the same. To make the picture more readable, we renumbered the variables so that the block structure is visible. For comparison, in Fig. \ref{pic1} we present the heat map of data matrix $U$ (rows and columns are permuted in the same way as in Fig. 
	\ref{pic1}(a)). 	We can interpret this result as an indication of hidden symmetry in genes and evidence that our procedure can be used as an	exploratory tool for finding such symmetries.
	
	Finally, we carry out the heuristic procedure introduced in Section 1.2$_{mf}$ for finding an RCOP model when the true graph is not complete. We threshold the entries of the partial correlation matrix at the level $\alpha=0.15$ and obtain a connected graph with $925$ edges, see Figure \ref{pic2}. The largest clique consists of $12$ vertices.
	\begin{figure}
		\includegraphics[width=0.45\textwidth]{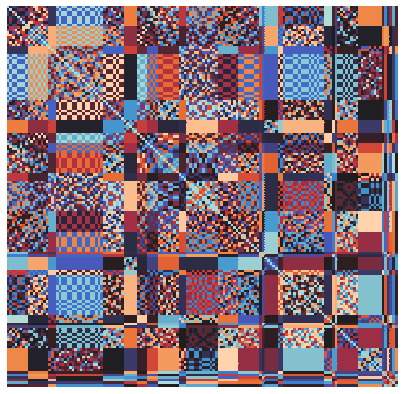}\includegraphics[height=6.23cm,width=0.5\textwidth]{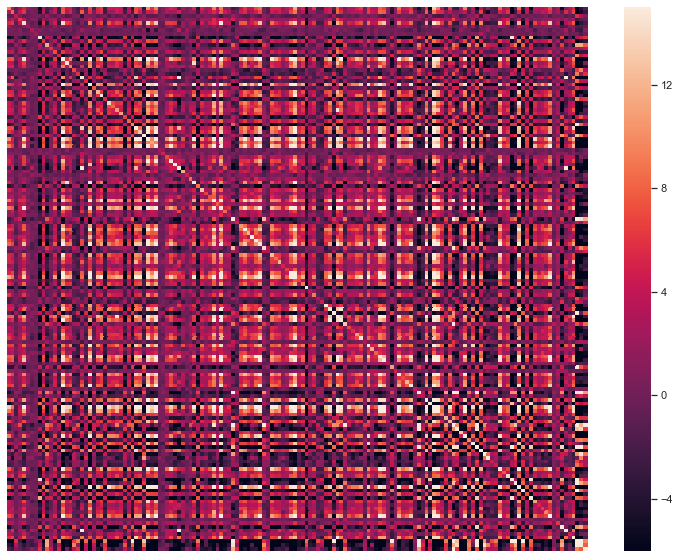}\\
		\makebox{{(a) \qquad\qquad\qquad\qquad\qquad\quad  \quad\qquad\qquad\qquad(b)}}
		\caption{(a) The color pattern of (permuted) space $\mathcal{Z}_{\hat{\Gamma}}$. (b) The heat map of (permuted) matrix $U$.}
		\label{pic1}
	\end{figure}
	
	\begin{figure}
		\includegraphics[width=\textwidth]{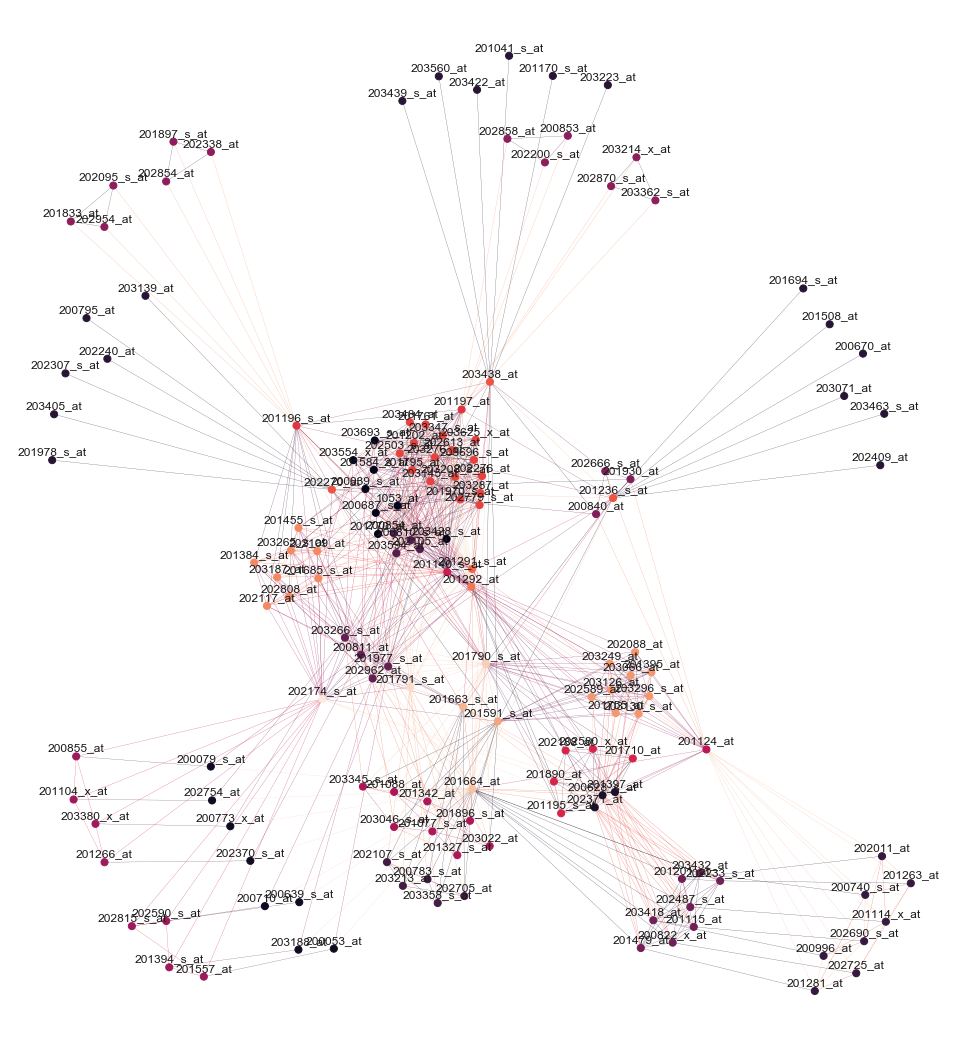}
		\caption{Graph corresponding to thresholded ($\alpha=0.15$) partial correlation matrix.}
		\label{pic2}
	\end{figure}

	\end{supplement}



\end{document}